\documentclass[12pt, twoside, leqno]{article}
\usepackage{amsmath, amsfonts, amsthm}
\usepackage{amssymb}
\usepackage{enumerate}
\newtheorem{theo}{Theorem}[section]
\newtheorem{cor}[theo]{Corollary}
\newtheorem{lemma}[theo]{Lemma}
\newtheorem{prop}[theo]{Proposition}
\newtheorem{claim}[theo]{Claim}
\theoremstyle{definition}

\newtheorem{rmk}[theo]{Remark}
\newtheorem{expl}[theo]{Example}

\newtheorem*{settingA}{Setting 1}
\newtheorem*{settingB}{Setting 2}
\numberwithin{equation}{section}
\newcommand{\id}{\textrm{Id}}
\newcommand{\E}{\operatorname{\mathbb{E}}} 
\newcommand{\N}{\ensuremath{\mathbb{N}}}   
\newcommand{\Proba}{\operatorname{\mathbb{P}}} 
\newcommand{\R}{\ensuremath{\mathbb{R}}}   
\newcommand{\C}{\ensuremath{\mathbb{C}}}   
\newcommand{\longto}{\mathop{\longrightarrow}\limits}
\newcommand{\supp}{\mathop{\rm supp}} 
\newcommand{\lsc}{l.s.c.}  
\newcommand{\usc}{u.s.c.}  
\newcommand{\NormCond}{$(i)$}
\newcommand{\ConstCond}{$(ii)$}
\newcommand{\ConstCondA}{$(ii)_1$}
\newcommand{\ConstCondB}{$(ii)_2$}
\newcommand{\MeasuCond}{$(iii)$}
\newcommand{\SemiInt}{$(iv)$}
\newcommand{\SettA}{Setting~1}
\newcommand{\SettB}{Setting~2}

\title{Some extensions of the Pr\'ekopa--Leindler inequality using
 Borell's stochastic approach}
\date{}

\author{Dario Cordero-Erausquin and Bernard Maurey \\
\small{Institut de Math\'ematiques de Jussieu-PRG, UPMC (Paris 6)}}

\begin{document}
\maketitle
\renewcommand{\thefootnote}{}
\renewcommand{\thefootnote}{\arabic{footnote}}
\setcounter{footnote}{0}

\begin{abstract}
 We present an abstract form of the Pr\'ekopa--Leindler inequality that
includes several known ---and a few new--- related functional inequalities
on Euclidean spaces. The method of proof and also the formulation of the
new inequalities are based on Christer Borell's stochastic approach to
Brunn--Minkowski type inequalities. 
\end{abstract}

\section{Introduction and main statement}\label{intro}
The Brunn--Minkowski inequality asserts that for Borel subsets $A, B$ of
$\R^n$ and $t \in [0, 1]$, the volume of the Minkowski combination 
\[
 (1 - t) A + t B = \{ (1 - t) a + t b : a \in A, \, b \in B \}
\]
satisfies the inequality $|(1 - t) A + t B| \ge |A|^{1 - t} |B|^t$, where
$|E|$ denotes the volume of a Lebesgue-measurable subset $E$ of $\R^n$.
There is a long story of functional generalizations of this inequality,
that we do not recall here; let us just mention that Borell's 1975
paper~\cite{borell75} remains a milestone in the subject. A somewhat
definitive form is given by the following theorem.

\begin{theo}[Pr\'ekopa--Leindler inequality]
Let $t \in [0, 1]$ and let $f_0, f_1, g : \R^n \to \R \cup \{+\infty\}$ be
Borel functions such that, for every $x_0, x_1 \in \R^n$,
\[
     g \left( (1 - t) x_0 + t x_1 \right) 
 \le (1 - t) f_0(x_0) + t f_1(x_1).
\]
Then
\[
     \int_{\R^n} e^{-g(x)} \, dx 
 \ge \left(\int_{\R^n} e^{-f_0(x)} \, dx \right)^{1 - t}
     \left(\int_{\R^n} e^{-f_1(x)} \, dx \right)^t.
\]
\end{theo}

 Accepting the value $+\infty$ enables us to reach directly indicator
functions $\mathbf{1}_E = e^{-f_E}$, by letting $f_E$ be $0$ on $E$ and
$+\infty$ outside. However, we can restrict ourselves to Borel functions
that are not Lebesgue-almost everywhere equal to $+\infty$. Such Borel
functions $f$ will be said to be \emph{L-proper}, which is equivalent to
saying that $\int_{\R^n} e^{ - f(x)} \, dx > 0$. 

 One can reasonably argue that the interest of the Pr\'ekopa--Leindler
inequality resides not only in its consequences, which are numerous (some
are recalled for instance in~\cite{gardner, Maurey:2005p220}), but also in
the emphasis it has put on log-concavity, and in the related techniques of
proof it has originated, such as mass transportation or semi-group
techniques, and more recently $L^2$-methods (as in~\cite{CEonB}). 

 Here, we will concentrate on Borell's stochastic
approach~\cite{Borell:2000p204} to the inequality above. It somehow reduces
the inequalities under study to the convexity of $|\cdot|^2$, the square of
the Euclidean norm on $\R^n$. It will allow us to obtain some unexpected
inequalities, for instance that of the following proposition. 

\begin{prop}\label{propM}
Let $f_0, f_1$, $g_0, g_1$ be four Borel functions from $\R^n$ to
$\R \cup \{+\infty\}$ such that, for every $x_0, x_1 \in \R^n$,
\begin{equation} \label{FirstBasic}
     g_0( 2 x_0 / 3 + x_1 / 3) + g_1( x_0 / 3 + 2 x_1 / 3) 
 \le f_0(x_0) + f_1(x_1).
\end{equation}
Then
\[
     \left(\int_{\R^n} e^{-g_0(x)} \, dx \right)  
     \left(\int_{\R^n} e^{-g_1(x)} \, dx \right)
 \ge \left(\int_{\R^n} e^{-f_0(x)} \, dx \right)  
     \left(\int_{\R^n} e^{-f_1(x)} \, dx \right).
\]
\end{prop}

 We will see that it is rather natural to arrive at this type of inequality
using Borell's stochastic approach, whereas it seems not to be the case
with other methods, for instance those based on transportation methods. The
point here is to split the values of the functions $f_i$'s at some points
into the values of two functions $g_j$'s at some related points. This is
not interesting in the case where the functions $f_i$'s take only the values
$+\infty$ and $0$, and the functional inequality above does not give
anything new when applied to the case where the functions $e^{-f_i}$'s are
indicators of sets, as we will explain in Section~$4$ below. 

 The previous proposition and its proof suggest actually more general
inequalities.  Writing the conclusion as
\[
     \sum_j - \log \Bigl( \int_{\R^n} e^{-g_j} \Bigr)
 \le \sum_i - \log \Bigl( \int_{\R^n} e^{-f_i} \Bigr), 
\]
we may think about
an extension of the results, where the finite families of functions 
$f_i, g_j$ are replaced by families $f_s, g_t$ depending upon continuous
parameters $s, t$ (as for instance in~\cite{barthe04}). In the ``basic
assumption"~\eqref{FirstBasic}, the two values $x_0, x_1$ on the 
right-hand side of the inequality will be replaced, for example, by a
selection $x = \{x(s)\}_{s \in [0, 1]}$ of points of $\R^n$, and the values 
$y = \{y(t)\}_{t \in [0, 1]}$ on the left-hand side will be obtained from a
linear transformation $A$ acting on this data $x$, i.e. $y = Ax$. So,
roughly speaking, under appropriate assumptions on the operator $A$, we
may expect that if for ``all" $x = \{x(s)\}$ we have
\[
     \int_0^1 g_t ((Ax)(t)) \, dt 
 \le \int_0^1 f_s (x(s)) \, ds
\]
then it will follow that 
\[
     \int_0^1 - \log \Bigl( \int_{\R^n} e^{-g_t} \Bigr) \, dt
 \le \int_0^1 - \log \Bigl( \int_{\R^n} e^{-f_s} \Bigr) \, ds.
\] 
Actually, there was no reason, when replacing the sums by integrals, to use
the ``uniform" distributions $dt$ and $ds$ rather than probability measures
$\mu(dt)$ and $\nu(ds)$ on $[0, 1]$, which include the discrete case when
these measures are convex combinations of Dirac measures. Anyway, the main
question is to understand what the appropriate conditions are to impose on
the linear operator $A$.  

 Several points must be set before proceeding. We say that a real function
$F$ on a measure space $(\Omega, \Sigma, \mu)$ is 
\emph{$\mu$-semi-integrable} if at least one of $F^+ = \max(F, 0)$ or 
$F^- = \max(-F, 0)$ is $\mu$-integrable. The integral of $F$ takes then a
definite value in $[-\infty, +\infty]$. This assumption is needed for 
$F(s) = - \log \Bigl( \int_{\R^n} e^{-f_s} \Bigr)$ in order to make sense
of the preceding integrals.

 We introduce the following abstract setting. 

\begin{settingA}
We are given 
\begin{itemize} 
\item Two measure spaces $X_1 = (\Omega_1, \Sigma_1, \mu_1)$ and
$X_2 = (\Omega_2, \Sigma_2, \mu_2)$, where $\Sigma_i$ is a $\sigma$-algebra
of subsets of $\Omega_i$ and $\mu_i$ is a \emph{finite} measure, 
$i = 1, 2$. We assume that $\Omega_1$ is a Polish topological space, and
$\Sigma_1$ is its Borel $\sigma$-algebra.
\item An integer $n \ge 1$ and a continuous linear operator 
\[
 A : L^2(X_1, \R^n) \to L^2(X_2, \R^n),
\] 
where the $L^2$-norms of the $\R^n$-valued functions are computed with
respect to the Euclidean norm $|\cdot |$ on $\R^n$ and to the measures
$\mu_1$ and $\mu_2$, respectively. We assume that
\begin{enumerate}
\item[\NormCond] the operator $A$ satisfies the norm condition 
$\|A\| \le 1$,
\item[\ConstCond] the operator $A$ acts as identity operator on the constant
vector valued functions, {\it i.e.}, for any $v_0 \in \R^n$, the constant
function $\Omega_1 \ni s \mapsto v_0$ is sent by $A$ to the constant
function $\Omega_2 \ni t \mapsto v_0$.
\end{enumerate}

\item
Two families $\{ f_s \}_{s \in \Omega_1}$ and $\{ g_t \}_{t \in \Omega_2}$
of Borel functions from $\R^n$ to $\R \cup \{+\infty\}$ satisfying
\begin{enumerate}
\item[\MeasuCond] 
the functions $(s, x) \mapsto f_s(x)$ and $(t, x) \mapsto g_t(x)$ are
measurable with respect to the $\sigma$-algebras 
$\Sigma_i \otimes \mathcal{B}_{\R^n}$, $i = 1, 2$, respectively, the
notation $\mathcal{B}_{\R^n}$ being for the Borel $\sigma$-algebra 
of~$\R^n$,
\item[\SemiInt]
the functions
\[
 \Omega_1 \ni s \mapsto
  \log \Bigl( \int_{\R^n} e^{-f_s(x)} \, dx \Bigr)
 \ \ \text{and} \ 
 \Omega_2 \ni t \mapsto
  \log \Bigl( \int_{\R^n} e^{-g_t(x)} \, dx \Bigr)
\]
are semi-integrable with respect to $\mu_1$, $\mu_2$, respectively.
\end{enumerate}
\end{itemize} 
\end{settingA}

 Unfortunately, there is a tradeoff between the generality of the statement
we can reach and the technical assumptions that are required in the proof.
We will start with a fairly general situation. In most applications, the
technical assumptions on the functions in the statement below are either
easy to impose or to discard, as explained for instance in
Remark~\ref{rem:restrictions} and as shown in Theorem~\ref{discretePL1}
below.  So, our most abstract version of the Pr\'ekopa--Leindler inequality
reads as follows. 

\begin{theo}\label{genPL1}
Under \SettA\ with $\mu_1$ and $\mu_2$ having the same finite mass,
$\mu_1(\Omega_1) = \mu_2(\Omega_2) < +\infty$, we make the additional
assumptions on the functions: 
\begin{itemize}
\item[--]
for every $s \in \Omega_1$, the function $f_s$ is non-negative, and for
every $t \in \Omega_2$, the function $g_t$ is non-negative and lower
semicontinuous,
\item[--]
for some $\varepsilon_0 > 0$, we have that
\begin{equation}\label{dominated}
 \int_{\Omega_1} \log^- 
  \Bigl( 
   \int_{\R^n} \exp \bigl( - f_s(x) - \varepsilon_0 |x|^2 \bigr) \, dx 
  \Bigr) 
   \, d\mu_1(s) < +\infty.
\end{equation}
\end{itemize}
Then, if for every $\alpha \in L^2(X_1, \R^n)$ we have
\begin{equation}
\label{condPL}
     \int_{\Omega_2} g_t ( (A \alpha)(t) ) \, d\mu_2(t) 
 \le \int_{\Omega_1} f_s( \alpha(s) ) \, d\mu_1(s),
\end{equation}
it follows that
\begin{equation}
\label{conclPL}
     \int_{\Omega_2}
      -\log \left( \int_{\R^n} e^{-g_t(x)} \, dx \right)
       \, d\mu_2(t) 
 \le \int_{\Omega_1} 
      -\log \left( \int_{\R^n} e^{-f_s(x)} \, dx \right)
       \, d\mu_1(s).
\end{equation}
\end{theo}

\begin{rmk}\label{rem:restrictions}
We can easily relax the restrictions $f_s(x), g_t(x) \ge 0$. Suppose indeed
that $f_s(x) \ge - a(s)$, $g_t(x) \ge - b(t)$, where $a(s), b(t)$ are
non-negative functions on $\Omega_1, \Omega_2$ that are $\mu_1$,
$\mu_2$-integrable respectively. Assuming, as we may, that
$\int_{\Omega_1} a \, d\mu_1 = \int_{\Omega_2} b \, d\mu_2$, we see that
the ``basic assumption"~\eqref{condPL} and the conclusion~\eqref{conclPL}
are unchanged when passing from $f_s, g_t$ to the non-negative functions
$f_s + a(s)$, $g_t + b(t)$. So, in the theorem above, we are free to
assume only that the functions are bounded from below in the way
described in the latter discussion.
\end{rmk}

\begin{rmk}
By classical arguments of measure theory, we may replace $f_s$ Borel by 
$\widetilde f_s$ \lsc\ such that $f_s \le \widetilde f_s$ and 
$\int_{\R^n} e^{-\widetilde f_s} \simeq \int_{\R^n} e^{-f_s}$, but we do
not know how to do it on the $g_t$ side in a way that~\eqref{condPL}
remains true. However, when the family $\{ g_t \}$ consists of a single
function $g$, it is easy to replace $g$ by a \lsc\ function. The real issue
is with the values of the $f_s$ functions that are equal to $+\infty$. If
$f_s, g_t$ are merely non-negative Borel functions, with $f_s$ locally
bounded, it is possible to reduce the problem to continuous functions
$\{f_s\}$ and $\{g_t\}$ by using convolution with non-negative compactly
supported continuous kernels, hence reducing this case to the preceding
theorem.

 On the other hand, it is hard to believe that our \lsc\ assumption is
necessary for the validity of Theorem~\ref{genPL1}. We rather tend to think
that the result is true for general Borel functions (as we can prove it in
the ``discrete case" of Theorem~\ref{discretePL1}; see also
Remark~\ref{RmkGtAndFi}).
\end{rmk}

 Suppose that the conditions of \SettA\ and
Theorem~\ref{genPL1} are satisfied for $\{f_s\}$ and $\{g_t\}$. Using the
``norm condition"~\NormCond, we see that the basic
assumption~\eqref{condPL} remains true if we add the same multiple
$\varepsilon |x|^2$ of $|x|^2$, $\varepsilon > 0$, to all the functions
$f_s$ and $g_t$. We can see that after this addition, the other conditions
of \SettA\ and Theorem~\ref{genPL1} remain obviously true, with
two exceptions that are either less obvious or not always true: the
semi-integrability condition~\SemiInt\ remains true because $f_s$ and $g_t$
are also assumed to be non-negative, thus
$\int_{\R^n} e^{-f_s(x) - \varepsilon |x|^2} \, dx \le C(\varepsilon)$ (and
the same for $g_t$), and if we assume $0 < \varepsilon < \varepsilon_0$,
then the condition~\eqref{dominated} remains true, with $\varepsilon_0$
replaced by $\varepsilon_0 - \varepsilon$. It follows that the
conclusion~\eqref{conclPL} holds if we replace the inside integration with
respect to the Lebesgue measure on $\R^n$ by integration with respect to
the isotropic Gaussian probability measure $\gamma_{n, \tau}$ on $\R^n$
defined by $d\gamma_{n, \tau}(x)
 = e^{-|x|^2 / 2 \tau} \, dx / (2 \pi \tau)^{n/2}$, provided
$2 \tau > \varepsilon_0^{-1}$. Actually, the proof of Theorem~\ref{genPL1}
will start with the Gaussian case and obtain the Lebesgue measure case from
it, the Lebesgue case being the ``flat" extremal case when 
$\tau \to +\infty$. Note that, indeed, the (log of the) normalization
constant $(2 \pi \tau)^{n/2}$ on the two sides of~\eqref{conclPL}
cancels out since $\mu_1$ and $\mu_2$ have the same finite mass.
\medskip

 In many applications, $X_1$ and $X_2$ are finite probability spaces, and
then we work with finite families of objects parameterized by $\Omega_1$
and $\Omega_2$, or rather by the supports $\supp(\mu_1)$ and
$\supp(\mu_2)$; in particular, the mappings $\alpha : \Omega_1 \to \R^n$
are families of $|\Omega_1|$ vectors of $\R^n$ and the linear operator
\[
 A : (\R^n)^{|\Omega_1|} \to (\R^n)^{|\Omega_2|}
\]
is norm-one for the operator norm associated to the $\ell^2$-norms weighted
by the $\mu_i$'s, with the property that the vector 
$(x, \ldots, x) \in (\R^n)^{|\Omega_1|}$ is sent to 
$(x, \ldots, x) \in (\R^n)^{|\Omega_2|}$, for every $x \in \R^n$. 
In this case, we can relax the technical assumptions that were 
made in Theorem~\ref{genPL1}.

\begin{theo}\label{discretePL1}
Under \SettA, assume in addition that $\mu_1$ and $\mu_2$ are
measures with \emph{finite support} and with 
$\mu_1(\Omega_1) = \mu_2(\Omega_2) < +\infty$. Then the
assumption~\eqref{condPL} implies~\eqref{conclPL} with no further
restriction on the functions $\{ f_s \}$ and $\{ g_t \}$. 
\end{theo}

Let us examine the first situations that one encounters when $\Omega_1$ and
$\Omega_2$ are finite, with $\mu_1(\Omega_1) = \mu_2(\Omega_2) = 1$. 
\begin{itemize}

\item
The case $|\Omega_1| = |\Omega_2| = 1$ is trivial. One has necessarily 
$A(x) = x$ by~\ConstCond, and the statement amounts to the monotonicity of
the integral. 

\item Assume $|\Omega_1| = 1$ and $|\Omega_2| = 2$. So 
$\Omega_2 = \{0, 1\}$, say, and $\mu_2 = (1-t) \delta_0 + t \delta_1$, 
$t \in [0, 1]$. The map $A : \R^n \to \R^n \times \R^n$ ought to be 
$A(x) = (x, x)$, which satisfies both required conditions, and the
statement asserts that if the functions $g_0, g_1, f$ satisfy
$(1-t) g_0(x) + t g_1(x) \le f(x)$ then 
$\bigl( \int e^{-g_0} \bigr)^{1-t} \bigl( \int e^{-g_1} \bigr)^t
 \ge \int e^{-f}$. This is just H\"older's inequality. 

\item Assume $|\Omega_1| = 2$ and $|\Omega_2| = 1$. Here 
$\Omega_1 = \{0, 1\}$, say, and $\mu_1 = (1-t) \delta_0 + t \delta_1$, 
$t \in [0, 1]$. Then we can take $A : \R^n \times \R^n \to \R^n$ to be 
$A(x_0, x_1) = (1-t) x_0 + t x_1$ (actually this is the only possible
choice, as explained in Remark~\ref{noexotic}). This map will satisfy the
required assumptions, by the convexity of the norm squared. The statement
asserts that if the functions $g, f_0, f_1$ satisfy
$g((1-t) x_0 + t x_1) \le (1-t) f_0(x_0) + t f_1(x_1)$ for all 
$(x_0, x_1) \in \R^n \times \R^n$, then 
$\int e^{-g}
 \ge \bigl( \int e^{-f_0} \bigr)^{1-t} \bigl( \int e^{-f_1} \bigr)^t$. 
This is the Pr\'ekopa-Leindler inequality. 

\item Assume $|\Omega_1| = |\Omega_2| = 2$. This is the first case where
something new appears. Let us illustrate this by an example. Assume
$\Omega_1 = \Omega_2 = \{0, 1\}$ and 
$\mu_1 = \mu_2 = \frac12 \delta_0 + \frac12 \delta_1$. Consider the map 
$A : \R^n \times \R^n \to \R^n \times \R^n$ defined by
\[
 A(x_0, x_1) = ( 2 x_0 / 3 + x_1 / 3, x_0 / 3 + 2 x_1 / 3).
\]
We are in \SettA\ since $A(x, x) = (x, x)$ and 
$| 2 x_0 / 3 + x_1 / 3|^2 + | x_0 / 3 + 2 x_1 / 3|^2
 \le |x_0|^2 + |x_1|^2 $. The abstract statement then reduces exactly to
Proposition~\ref{propM} above. 

\end{itemize}

 In the next section, we present Borell's approach and establish in
Proposition~\ref{GaussianPropo} a Gaussian version of Theorem~\ref{genPL1};
the proof of this proposition can be considered as the heart of the present
paper. In Section~\ref{ProofMainTheo}, we present the proofs of 
Theorems~\ref{genPL1} and~\ref{discretePL1}. Then, in the following
section, we discuss some consequences of it. In Section~\ref{GeneBL} we
present a generalization of Theorem~\ref{genPL1}, when the functions $f_s$
and $g_t$ live on Euclidean spaces of different dimensions. The result will
include as particular cases the Brascamp--Lieb inequality (in the geometric
form) and its reverse form devised by Franck Barthe~\cite{Barthe:1998gu}.
Several technical proofs that have little geometric interest, and involve
mostly measure theoretic arguments, are gathered in Section~\ref{technical}.

\section{Borell stochastic approach and Gaussian inequality}
Borell's stochastic proof of the Pr\'ekopa--Leindler inequality relies on
the representation formula given in the next lemma. Let $(B_r)_{r \ge 0}$
be a standard Brownian motion with values in $\R^n$, starting at $0$, with
filtration $\mathcal F = ({\mathcal F}_r)_{r \ge 0}$. Let
$P_r = e^{r \Delta / 2}$, $r \ge 0$, be the heat semigroup on $\R^n$
associated with this Brownian motion,
\[
   (P_r f)(x) 
 = \E f(x + B_r)
 = \int_{\R^n} f(x + y) e^{ - |y|^2 / (2 r)} 
    \, \frac {dy} {(2 \pi r)^{n / 2}},
 \quad x \in \R^n,
\]
for $f$ bounded and continuous on $\R^n$ and $r > 0$. An $\R^n$-valued
drift $u = \{ u_r \}_{r \le T}$ will be called \emph{of class $D_2$} on
$[0, T]$ if it is $\mathcal F$-progressively measurable on $[0, T]$ and
\[
 \E \int_0^T |u_r|^2 \, dr < +\infty.
\]

\begin{lemma} \label{LBorell}
Let $T > 0$ be fixed. For every bounded continuous real function 
$f : \R^n \to \R$, we have
\begin{equation}
\label{borell}
   - \log P_T \left( e^{-f} \right)(0) 
 = \inf_u \E \left[ 
              f \Bigl( B_T + \int_0^T u_r \, dr \Bigr)
               + \frac 1 2 \int_0^T |u_r |^2 \, dr
             \right],
\end{equation}
where the infimum is taken over $\R^n$-valued drifts $\{ u_r \}_{r \le T}$
of class $D_2$. Moreover, the infimum is attained. 
\end{lemma}

\begin{proof}[Proof {\rm (see Borell~\cite{Borell:2000p204})}]
We begin by assuming that $f$ is bounded and has bounded derivatives of
order $\le 2$. For a drift $u = \{ u_r \}_{r \le T}$ of class $D_2$, we
define $X^u_r := B_r + \int_0^r u_\rho \, d\rho$, which satisfies the
stochastic differential equation 
\[
 X^u_0 = 0, \qquad dX_r^u = dB_r + u_r \, dr
\]
on the interval $[0, T]$. For $0 \le r \le T$, define $f_r = f^T_r$ by 
$e^{- f_r } = P_{T-r} e^{-f}$, that is to say
\begin{equation} \label{Fr}
 f_r (x) = -\log (P_{T-r} e^{-f})(x),
 \quad x \in \R^n.
\end{equation}
The function $(r, x) \mapsto f_r(x)$ on $(0, T) \times \R^n$ satisfies the
partial differential equation
\[
   \partial_r f_r
 = - \frac 1 2 \Delta f_r + \frac 1 2 |\nabla f_r |^2.
\]
By a direct application of the It\=o formula, we see that the process
\[
    M_r
 := f_r(X_r^u) + \frac 1 2 \int_0^r |u_\rho|^2 \, d\rho,
 \qquad r \in [0, T],
\]
is a submartingale for any drift $u$ in $D_2$, since
\[
   dM_r
 = \nabla f_r (X_r^u) \cdot dB_r
    + \frac 1 2 \bigl| \nabla f_r (X_r^u) + u_r \bigr|^2 dr.
\]
This implies the ``inequality case" of~\eqref{borell}, namely, an upper
bound of $-\log (P_T e^{-f})(0)$ for every drift $u$ in $D_2$. Indeed, note
that $M_0 = f_0(X_0) = -\log P_T(e^{-f})(0)$ and $f_T = f$, and so the
inequality in~\eqref{borell} immediately follows by considering $\E M_r$ at
$r = 0$ and $r = T$. Moreover, if $u_r = -\nabla f_r (X_r^u)$, {\it i.e.}, 
if $X_r^u$ is the process solving the stochastic differential equation
\begin{equation}
 \label{sde}
 X_0 = 0, \quad
 dX_r = dB_r - \nabla f_r (X_r) \, dr,
\end{equation}
then $M_r$ becomes a martingale, thus giving an equality case
in~\eqref{borell}.

 Assume now that $f$ is bounded and continuous on $\R^n$. Adding a constant
to $f$, we may suppose that $f \ge 0$. Define $f_r$ as in~\eqref{Fr}, and
note that $f_r$ is bounded above and $\ge 0$. Since $f$ is bounded, 
$P_r e^{-f}$ is bounded away from $0$, so the function 
$f_{T - \varepsilon} = - \log \bigl( P_\varepsilon e^{-f} \bigr)$ has
bounded derivatives of all orders, for every $\varepsilon \in (0, T]$.
Writing $T_\varepsilon = T - \varepsilon$ and
$  e^{-f_r} 
 = P_{T_\varepsilon - r} (P_\varepsilon e^{-f})
 = P_{T_\varepsilon - r} (e^{-f_{T_\varepsilon}})$, 
$0 \le r \le T_\varepsilon$, we are back to the ``good setting" on 
$[0, T_\varepsilon]$. Hence the optimal representation by the martingale
\[
   M_r 
 = f_r \Bigl( B_r + \int_0^r u_\rho \, d\rho \Bigr)
    + \frac 1 2 \int_0^r |u_\rho|^2 \, d\rho
 = M_0 + \int_0^r \nabla f_\rho(X_\rho) \cdot dB_\rho,
\]
with $M_0 = f_0(0)$, $u_r = - \nabla f_r(X_r)$ and 
$X_r = B_r + \int_0^r u_\rho \, d\rho$, is valid for $r \le T_\varepsilon$.
Since $f_r \ge 0$, we see that $\int_0^r |u_\rho|^2 \, d\rho \le 2 M_r$ and
we obtain that
\[
     \E |M_r - M_0|^2 
  =  \E \Bigl( \int_0^r |\nabla f_\rho(X_\rho)|^2 \, d\rho \Bigr)
  =  \E \Bigl( \int_0^r |u_\rho|^2 \, d\rho \Bigr)
 \le 2 \E M_r
  =  2 f_0(0). 
\]
We have $T^{-1} \E \bigl( \int_0^T |u_\rho| \, d\rho \bigr)^2
 \le \E \int_0^T |u_\rho|^2 \, d\rho \le 2 f_0(0)$, hence $u \in D_2$,
$X_r$ converges a.s. and in $L^2$ to 
$X_T = B_T + \int_0^T u_\rho \, d\rho$. We see that $(M_r)_{r < T}$ is an
$L^2$-bounded martingale, thus $M_r$ converges in $L^2$-norm, as $r \to T$,
to a limit $M_T$ such that $\E M_T = M_0$. On the other hand, $M_r$
converges almost surely to
\begin{equation} \label{converges}
 f \Bigl( B_T + \int_0^T u_\rho \, d\rho \Bigr)
  + \frac 1 2 \int_0^T |u_\rho|^2 \, d\rho
\end{equation}
when $r \to T$ since $f_r$ converges locally uniformly to the bounded
continuous function $f$. This implies that $M_T$ is equal to the expression
in~\eqref{converges}, hence the expectation of~\eqref{converges} is
equal to $M_0 = f_0(0)$, that is, it implies formula~\eqref{borell} with
equality. 

 In the inequality case of~\eqref{borell}, the drift $u$ is in $D_2$ by
assumption, hence $(X_r^u)_{r < T}$ defined as above converges a.s. and in
$L^2$ to $X_T^u$, as $r \to T$. The submartingale $(M_r)_{r < T}$ is
$L_2$-bounded, hence converges a.s. and in $L^2$ to the expression
in~\eqref{converges}, and the result follows.
\end{proof}

\begin{rmk}
If $f$ is bounded and upper semicontinuous on $\R^n$, then for every 
$x \in \R^n$ and $\varepsilon > 0$, there is $r_0 < T$ and a neighborhood
$V$ of $x$ such that $(P_{T-r} e^{-f})(y) > e^{-f(x) - \varepsilon}$ for 
$y \in V$, $r_0 < r < T$, that is to say, 
$(a) : f_r(y) < f(x) + \varepsilon$. When $f$ is lower semicontinuous, the
inequality is reversed, $(b) : f_r(y) > f(x) - \varepsilon$. When $f$ is
\usc\ it follows from $(a)$ that in the inequality case
of~\eqref{borell}, the limit of $(M_r)_{r < T}$ as $r \to T$ is less than
or equal to the expression in~\eqref{converges}, so the inequality case
remains true. For every bounded Borel function $f$, we can find 
$\widetilde f$ \usc\ such that $\widetilde f \le f$ and
$\int_{\R^n} e^{- \widetilde f} \, d\gamma
 \simeq \int_{\R^n} e^{- f} \, d\gamma$. This implies that the inequality
case in~\eqref{borell} is true for every bounded Borel function $f$. Using
the reverse inequality $(b)$ for \lsc\ functions, we see that the equality
in~\eqref{borell} remains true for bounded lower semicontinuous functions
(but the infimum need not be achieved by an optimal martingale).
\end{rmk}

 We shall need more than the case of bounded functions. The proof of the
next lemma will be given in Section~\ref{technical}; it requires a closer
look at the argument given above.

\begin{lemma} \label{ExpoBoundLemma}
Formula~\eqref{borell} remains valid when $f$ is continuous, bounded
below and satisfies an exponential upper bound of the form
\begin{equation} \label{ExponeBound}
 f(x) \le a e^{b |x|},
 \quad a, b \ge 0,
 \quad x \in \R^n.
\end{equation}
The ``inequality case" in~\eqref{borell} is valid for any bounded below
continuous function $f$.
\end{lemma}

 The next proposition provides a rather direct and simple link between
Borell's lemma and our Theorem.

\begin{prop} \label{GaussianPropo}
Under \SettA, assume that $f_s$, $g_t$ are continuous $\ge 0$ on $\R^n$
(or bounded from below as in Remark~\ref{rem:restrictions}), and that $f_s$
satisfies the following exponential bound: there are non-negative
measurable functions $a(s), b(s)$ on $\Omega_1$ such that
\begin{equation}\label{exponebound}
 f_s(x) \le a(s) e^{ b(s) |x|},
 \quad x \in \R^n.
\end{equation}
Then, for every isotropic Gaussian probability measure 
$\gamma = \gamma_{n, \tau}$ on $\R^n$, the basic assumption~\eqref{condPL}
implies that
\begin{equation} \label{conclGL}
     \int_{\Omega_2} \!
      -\log \left( \int_{\R^n} e^{-g_t(x)} \, d\gamma(x) \right)
       \, d\mu_2(t) 
 \le \int_{\Omega_1} \! 
      -\log \left( \int_{\R^n} e^{-f_s(x)} \, d\gamma(x) \right)
       \, d\mu_1(s).
\end{equation}
\end{prop}

 As seen from the proof below of Theorem~\ref{discretePL1} (which indeed
reduces the study to the situation of this Proposition), we can remove the
additional assumptions on $f_s$ and $g_t$ when the measures $\mu_1$ and
$\mu_2$ have finite support. 

\begin{proof}
The proof will be an application of formula~\eqref{borell}, as extended
in Lemma~\ref{ExpoBoundLemma}, to the functions $f_s$ and $g_t$. We may
assume that the right-hand side of~\eqref{conclGL} is not equal to
$+\infty$. The Gaussian measure $\gamma = \gamma_{n, \tau}$ is the
distribution of $B_T$ for $T = \tau > 0$, so that
\[
   - \log \Bigl( \int_{\R^n} e^{- f_s(x)} \, d\gamma(x) \Bigr)
 = - \log (P_T e^{-f_s})(0) 
 < +\infty,
 \quad
 \mu_1\text{-a.e.}
\]
Let $\varepsilon > 0$ be given. By Lemma~\ref{ExpoBoundLemma}, for almost
every $s \in \Omega_1$, we can introduce $U(s) \in D_2$, an almost optimal
drift for the function $f_s$, $s \in \Omega_1$, namely, a drift
$U(s) = \{ U_r(s) \}_{0 \le r \le T}$ such that
\begin{equation} \label{DefiDrift}
   \E \Bigl[ f_s \Bigl( B_T + \int_0^T U_r(s) \, dr \Bigr)
        + \frac 1 2 \int_0^T |U_r(s)|^2 \, dr \Bigr]
 < - \log (P_T e^{-f_s})(0) + \varepsilon.
\end{equation}
We will show later (Claim~\ref{measurability}) that this process 
$\{ U(s) \}_{s \in \Omega_1}$ can be chosen to be $\Sigma_1$-measurable. 
Note also that~\eqref{DefiDrift} and $f_s \ge 0$ ensure that
\begin{align}
     &\int_{\Omega_1}
      \E \int_0^T |U_r(s)|^2 \, dr \, d\mu_1(s)
 \label{pourquoiL2}
 \\
 \le & \, 2 \Bigl( \int_{\Omega_1} - \log (P_T e^{-f_s})(0) \, d\mu_1(s)
          + \varepsilon \mu_1(\Omega_1) \Bigr)
  <  + \infty.
 \notag
\end{align}
Assume that the Brownian motion $(B_r)$ is defined on a probability space 
$(E, \mathcal{A}, \Proba)$. We have a $\R^n$-valued random process $U$,
that will be denoted by one of
\[
   U(s, r, \omega)
 = U_r(s)(\omega) 
 = U_r(s, \omega) 
 = U_{r, \omega}(s),
 \quad s \in \Omega_1, \, r \in [0, T], \, \omega \in E.
\] 
By~\eqref{pourquoiL2}, we know that
\[
   U \in L^2(\Omega_1 \times [0, T] \times E, \R^n)
 = L^2([0, T] \times E, L^2(\Omega_1, \R^n)). 
\]
We shall estimate
$P_T(e^{-g_t})$ using the inequality case of formula~\eqref{borell} given
by Lemma~\ref{ExpoBoundLemma}, with the drift
$\{ V_r(t) \}_{r \le T} = \{ A U_r(t) \}_{r \le T}$, namely
\begin{equation} \label{OneEq}
    V_{r, \omega} (t) 
 := A (U_{r, \omega}) (t)
  = A \bigl( s \mapsto U_{r,\omega}(s) \bigr) (t)
 \in \R^n,
 \quad
 t \in \Omega_2,
\end{equation}
where $U_{r, \omega} \in L^2(X_1, \R^n)$ for almost every $r, \omega$. We
shall use the basic assumption~\eqref{condPL} on the families $\{ f_s \}$
and $\{ g_t \}$ for the random function
$\alpha_\omega = B_T(\omega) + \beta_\omega$, where $\beta_\omega$ is
defined by
\begin{equation} \label{TwoEq}
    \beta_\omega (s)
 := \int_0^T U_r(s, \omega) \, dr,
\end{equation}
and where we consider $B_T(\omega)$ as a constant function of the $s$
variable. We know by~\eqref{pourquoiL2} that $\beta_\omega$ (and
$\alpha_\omega$) are in $L^2(X_1, \R^n)$ for almost every $\omega$. The
constant functions condition~\ConstCond\ on $A$ ensures that
\begin{align}
   (A \alpha_\omega)(t)
 &= A (B_T(\omega) + \beta_\omega) (t) 
 = B_T(\omega) + (A \beta_\omega) (t) 
 \label{Adrift}
 \\
 &= B_T(\omega) + \int_0^T V_r (t, \omega) \, dr.
 \notag
\end{align}

 As we said, we apply the inequality case of formula~\eqref{borell} for
$g_t$ with the drift $\{V_r\}$ and then we integrate in $t$, in order to
get that 
\begin{align*}
     &\int_{\Omega_2} - \log \bigl( P_T e^{- g_t} \bigr)(0) \, d\mu_2(t)
 \\
 \le &\int_{\Omega_2} \E \left[ g_t \Bigl(
                       B_T + \int_0^T V_r(t) \, dr \Bigr)
                 + \frac 1 2 \int_0^T |V_r(t) |^2 \, dr \right] \, d\mu_2(t).
\end{align*}
For future reference, we state an obvious consequence of~\eqref{Adrift},
\begin{align}
    g_t \bigl( (A \alpha_\omega) (t) \bigr)
  &= g_t \bigl( A (B_T(\omega) + \beta_\omega) (t) \bigr)
 \label{IsEnough}
   = g_t \bigl( B_T(\omega) + (A \beta_\omega) (t) \bigr)
 \\
  &= g_t \Bigl( B_T(\omega) + \int_0^T V_r (t, \omega) \, dr \Bigr).
 \notag
\end{align}
Using~\eqref{condPL} and~\eqref{IsEnough}, we have on the set $E$ the
pointwise inequality
\[
     \int_{\Omega_2} \! \!
      g_t \Bigl( B_T(\omega)
       + \int_0^T \! V_r(t, \omega) \, dr \Bigr) d\mu_2(t)
 {\le} \int_{\Omega_1} \! \!
      f_s \Bigl( B_T(\omega)
       + \int_0^T \! U_r(s, \omega) \, dr \Bigr) d\mu_1(s),
\]
and since $\|A\| \le 1$, we have another pointwise inequality, for every 
$r \in [0, T]$,
\begin{equation} \label{FourEq}
     \int_{\Omega_2} |V_r(t, \omega)|^2 \, d\mu_2(t) 
 \le \int_{\Omega_1} |U_r(s, \omega)|^2 \, d\mu_1(s).
\end{equation}
Finally
\begin{align*}
     &\int_{\Omega_2} - \log \bigl( P_T e^{- g_t} \bigr)(0) \, d\mu_2(t)
 \\
 \le &\int_{\Omega_1} 
      \E \left[ f_s \Bigl( B_T + \int_0^T U_r(s) \, dr \Bigr)
                 + \frac 1 2 \int_0^T |U_r(s) |^2 \, dr \right] \, d\mu_1(s)
\\
  <  &\int_{\Omega_1} - \log \bigl( P_T e^{- f_s} \bigr)(0) \, d\mu_1(s)
      + \varepsilon \mu_1(\Omega_1).
\end{align*}
We conclude by letting $\varepsilon \to 0$.

As the reader has noticed, the proof is rather short, provided one has put
forward the abstract properties contained in the four
equations~\eqref{OneEq}, \eqref{TwoEq}, \eqref{IsEnough} and~\eqref{FourEq}
that allow us to run Borell's argument.
\end{proof}

\section{Proofs of Theorems~\ref{genPL1} and~\ref{discretePL1} }%
\label{ProofMainTheo}

The proofs rely on Proposition~\ref{GaussianPropo}.  We shall explain how
to pass from a Gaussian measure to the Lebesgue measure, and how to
approximate our functions in order to meet the required technical
assumptions (continuity and lower/upper bounds).

\subsection{Proof of Theorem~\ref{discretePL1}}

Under \SettA\ in the discrete case, assume that 
$f_0, f_1, \ldots, f_p$ and $g_0, g_1,$ $\ldots,$ $g_q$ are Borel functions
from $\R^n$ to $\R \cup \{+\infty\}$, and that for all points 
$x_0, x_1, \ldots, x_p$ in $\R^n$ we have
\begin{equation} \label{hypote}
     \sum_{j=0}^q \nu_j g_j \Bigl( \sum_{i=0}^p a_{j, i} x_i \Bigr) 
 \le \sum_{i=0}^p \mu_i f_i(x_i),
\end{equation}
where $\mu_i, \nu_j > 0$, $\sum_{i=0}^p \mu_i = \sum_{j=0}^q \nu_j = 1$,
where the matrix $A = (a_{j, i})$ satisfies the norm condition~\NormCond\
and the constant functions condition~\ConstCond, each $a_{j, i}$ being
a $n \times n$ matrix acting from $\R^n$ to~$\R^n$. We want to prove that 
\begin{equation} \label{resulte}
     \sum_{j=0}^q - \nu_j \log
      \Bigl( \int_{\R^n} e^{- g_j(x)} \, d x \Bigr)
 \le \sum_{i=0}^p - \mu_i \log
      \Bigl( \int_{\R^n} e^{- f_i(x)} \, d x  \Bigr).
\end{equation}
If $\int_{\R^n} e^{-f_{i_0}} = 0$ for one $i_0$, then by the
semi-integrability assumption~\SemiInt, the other $\int_{\R^n} e^{-f_i}$
are finite, the right-hand side of~\eqref{resulte} is $+\infty$ and this
case is obvious. We may therefore assume that all $f_i$ are L-proper. It
follows that all $g_j$, $j = 0, \ldots, q$, are L-proper. Indeed, each set
$A_i = \{ f_i < +\infty \}$ has positive Lebesgue measure, so we can find
and fix a Lebesgue density point $x_i^{(0)}$ of $A_i$, for 
$i = 0, \ldots, p$. Then the set $U$ of $u \in \R^n$ such that 
$x_i^{(0)} + u \in A_i$ for all $i = 0, \ldots, p$, \textit{i.e.},
$U = \bigcap_{i=0}^p (A_i - x_i^{(0)})$, has positive measure.
We know that $\sum_{i=0}^p a_{j, i} = I_n$ since $A$ preserves constant
functions by~\ConstCond. We see from~\eqref{hypote} that 
\begin{align*}
     \sum_{j= 0}^q \nu_j g_j \Bigl( u + \sum_{i=0}^p a_{j, i} x_i^{(0)} \Bigr) 
  &=  \sum_{j= 0}^q \nu_j 
        g_j \Bigl( \sum_{i=0}^p a_{j, i} (x_i^{(0)} + u) \Bigr) 
 \\
 &\le \sum_{i = 0}^p \mu_i f_i(x_i^{(0)} + u)
  <  + \infty
\end{align*}
for every $u \in U$, hence all $g_j$ are L-proper. 

 For $\varepsilon > 0$ and $i = 0, \ldots, p$, $j = 0, \ldots, q$, the
functions $f_{i, \varepsilon}(x) = f_i(x) + \varepsilon |x|^2$, 
$g_{j, \varepsilon}(x) = g_j(x) + \varepsilon |x|^2$, still 
satisfy~\eqref{hypote} by the assumption $\|A\| \le 1$. We may reduce the
problem to proving~\eqref{resulte} for $f_{i, \varepsilon}$, 
$g_{j, \varepsilon}$, since we can pass to the limit as $\varepsilon \to 0$
in $\int_{\R^n} e^{-g_j(x) - \varepsilon x^2} \, dx$,
$\int_{\R^n} e^{-f_i(x) - \varepsilon x^2} \, dx$, 
obtaining~\eqref{resulte} by monotone convergence. In other words, we may
keep $f_i, g_j$ but replace the Lebesgue measure in~\eqref{resulte} by a
Gaussian measure $d\gamma(x) = e^{- \varepsilon |x|^2} \, dx$. If 
$\int_{\R^n} e^{-g_{j_0}} \, d \gamma = +\infty$ for one $j_0$, there is
nothing to prove: the other integrals $\int_{\R^n} e^{-g_{j}} \, d \gamma$,
$j \ne j_0$, are $> 0$ and the left-hand side of~\eqref{resulte} is
$-\infty$. Otherwise, for $N \in \N$, define $g_{j, N} = \min(g_j, N)$ and 
$f_{i, N} = \max(f_i, -N)$, that trivially satisfy~\eqref{hypote}, and
observe that it is enough to give the proof for $f_{i, N}$, $g_{j, N}$;
indeed, since $e^{ - g_{j, 0}} = e^{- \min(g_j, 0)} \le e^{-g_j} + 1$ is
integrable with respect to $d\gamma$, we may again pass to the (decreasing)
limit in the integrals $\int_{\R^n} e^{-g_{j, N}} \, d\gamma$ as 
$N \to +\infty$, and use monotone convergence for 
$\int_{\R^n} e^{-f_{i, N}} \, d\gamma$.

 Now we reduced the question to the case $g_j \le N$ and $f_i \ge - N$.
Thanks to the discrete situation, we may further assume that $f_{i_0}$ is
bounded above by $2 N / \mu_{i_0}$ for each $i_0$, since
\[
     \sum_{j=0}^q \nu_j g_j \Bigl( \sum_{i=0}^p a_{j, i} x_i \Bigr) 
      - \sum_{i \ne i_0} \mu_i f_i(x_i)
 \le N + (1 - \mu_{i_0}) N
  <  2 N,
\]
and for the same reason, we may also assume $g_{j_0}$ bounded below by 
$-2N / \nu_{j_0}$, while still keeping~\eqref{hypote} true. 

 Finally, we restricted the problem to bounded Borel functions $f_i, g_j$
and a Gaussian measure $\gamma$. If all $f_i, g_j$ are translated by the
same vector, then~\eqref{hypote} remains true by the constant functions
condition~\ConstCond, therefore~\eqref{hypote} is stable by convolution
with non-negative kernels. We may approximate in $L^1(\gamma)$-norm the
functions $f_i, g_j$ by convolution with a compactly supported continuous
non-negative kernel, keeping~\eqref{hypote} and meeting the assumptions of
Proposition~\ref{GaussianPropo}. We obtain thus, for a sequence of
continuous approximations $f_{i, k}$ and $g_{j, k}$, $k \in \N$, the
inequality 
\[
     \sum_{j=0}^q - \nu_j \log
      \Bigl(\int_{\R^n} e^{- g_{j, k}} \, d\gamma \Bigr)
 \le \sum_{i=0}^p - \mu_i \log
      \Bigl(\int_{\R^n} e^{- f_{i, k}} \, d\gamma \Bigr).
\]
Some subsequences of the sequence of continuous approximations tend almost
everywhere to $f_i, g_j$ respectively, and we finish with the
latter restricted problem by using the dominated convergence theorem.

This ends the proof of Theorem~\ref{discretePL1}.

\subsection{Proof of Theorem~\ref{genPL1}}

Suppose that we try to obtain the conclusion~\eqref{conclPL}
of~Theorem~\ref{genPL1} for $\{f_s\}$ and $\{g_t\}$ as a limit case of
``good cases" for which the conclusion is known, say $\{f_{s, k} \}$ and
$\{g_{t, k} \}$, $k \in \N$, such that $f_{s, k} \to f_s$, 
$g_{t, k} \to g_t$ when $k \to +\infty$. The next elementary lemma will
help us to do it. Consider a measure space $(\Omega, \Sigma, \mu)$ and a
measure $\nu$ on $(\R^n, \mathcal{B}_{\R^n})$, where $\mu$ and $\nu$ are
$\sigma$-finite. Let $h_{s, k}(x)$, $k \in \N$ and $h_s(x)$, 
$s \in \Omega$, $x \in \R^n$ be 
$\Sigma \otimes \mathcal{B}_{\R^n}$-measurable from $\Omega \times \R^n$ to 
$\R \cup \{+\infty\}$. Define $H_k$ and $H$ by 
\begin{equation} \label{FunctionH}
 e^{-H_k(s)} = \int_{\R^n} e^{-h_{s, k}(x)} \, d\nu(x),
 \quad
 e^{-H(s)} = \int_{\R^n} e^{-h_{s}(x)} \, d\nu(x).
\end{equation}
These functions are $\Sigma$-measurable by the theory behind Fubini's
theorem. Note that
\[
 H^+(s) = \log^- \Bigl( \int_{\R^n} e^{-h_{s}(x)} \, d\nu(x) \Bigr),
 \quad
 H^-(s) = \log^+ \Bigl( \int_{\R^n} e^{-h_{s}(x)} \, d\nu(x) \Bigr).
\]

\begin{lemma} \label{LimitsLemma}
Let $H_k(s), H(s)$ be defined by~\eqref{FunctionH}. Assume that $H_k$, 
$k \in \N$ and $H$ are $\mu$-semi-integrable.
\begin{itemize}
\item[$(a)$] Suppose that $h_{s, k}(x) \nearrow h_s(x)$ pointwise on
$\Omega \times \R^n$, and that $H_0(s) > - \infty$ for $\mu$-almost every
$s \in \Omega$. Then:
\begin{itemize}
\item[$(a1)$] $H_k(s) \!\nearrow\! H(s)$ $\mu$-a.e. and
$    \int_\Omega H(s) \, d\mu(s)
 {\ge} \limsup_k \int_\Omega H_k(s) \, d\mu(s)$.
\item[$(a2)$] If in addition $\int_\Omega H^-_0(s) \, d\mu(s) < +\infty$,
then 
\[
 \int_\Omega H(s) \, d\mu(s) = \lim_k \int_\Omega H_k(s) \, d\mu(s).
\]
\end{itemize}
\item[$(b)$] Suppose that $h_{s, k}(x) \searrow h_s(x)$ pointwise on
$\Omega \times \R^n$. Then:
\begin{itemize}
\item[$(b1)$] $H_k(s) \searrow H(s)$ pointwise on $\Omega$, and
\[
 \int_\Omega H(s) \, d\mu(s) \le \liminf_k \int_\Omega H_k(s) \, d\mu(s).
\]
\item[$(b2)$] If in addition $\int_\Omega H^+_0(s) \, d\mu(s) < +\infty$,
then 
\[
 \int_\Omega H(s) \, d\mu(s) = \lim_k \int_\Omega H_k(s) \, d\mu(s).
\]
\end{itemize}
\end{itemize}
\end{lemma}

\begin{proof} Apply successively the classical results of Integration
Theory: the Fatou lemma, the monotone convergence theorem and the dominated
convergence theorem. In case~$(a)$, we have 
$e^{-h_{s, k}(x)} \searrow e^{- h_s(x)}$, and $H_0(s) > - \infty$ allows us 
to apply dominated convergence to $\nu$ and deduce that 
$H_k(s) \nearrow H(s)$. Next, $H^+_k(s) \nearrow H^+(s)$ and
$H^-_k(s) \searrow H^-(s)$. In $\int_\Omega H
 = \int_\Omega H^+ - \int_\Omega H^-$ (that makes sense by the
semi-integrability assumption), apply monotone convergence for $H^+$ and
Fatou for $H^-$. If $\int_\Omega H^-_0 \, d\mu < +\infty$, replace Fatou by
dominated convergence for $H^-$. The proof of~$(b)$ is similar and left to
the reader.
\end{proof}

\begin{proof}[Proof of Theorem~\ref{genPL1}]
Since $\mu_1$ and $\mu_2$ have the same finite mass it is enough, as
mentioned after the statement of the theorem, to have an inequality
involving $(P_T e^{-g_t})(0)$ and
\[
   (P_T e^{-f_s})(0)
 = \int_{\R^n} e^{-f_s(x) - |x|^2 / (2 T)} 
    \, \frac {dx} { (2 \pi T)^{n/2} },
 \quad
 2 T > 1 / \varepsilon_0.
\] 
The result will follow by letting $T \to +\infty$. Indeed, 
$f_s(x) + \varepsilon |x|^2$ decreases to $f_s(x)$, 
$g_t(x) + \varepsilon |x|^2$ decreases to $g_t(x)$ when 
$\varepsilon \searrow 0$; for $\varepsilon \in [0, \varepsilon_0]$, define
$F_\varepsilon(s)$, $G_\varepsilon(t)$ as in~\eqref{FunctionH}, with
$\nu$ being the Lebesgue measure, that is to say, let
$e^{-F_\varepsilon(s)}
 = \int_{\R^n} e^{- f_s(x) - \varepsilon |x|^2} \, dx$ and use the similar
expression for $G_\varepsilon(t)$. By the semi-integrability
condition~\SemiInt, $G_0(t)$ is semi-integrable, and we may assume that
$G^-_0(t)$ is integrable, otherwise the left-hand side of~\eqref{conclPL}
is $- \infty$, an obvious case. We also know by the
assumption~\eqref{dominated} that 
$  F^+_{\varepsilon_0}(s) 
 = \log^- \bigl( \int_{\R^n} e^{- f_s(x) - \varepsilon_0 |x|^2} \, dx
   \bigr)$
is integrable. This yields that $F_\varepsilon(s)$ and $G_\varepsilon(t)$
are semi-integrable for $\varepsilon \in [0, \varepsilon_0]$. Use
Lemma~\ref{LimitsLemma}, $(b)$ to see that $F_\varepsilon(s) \to F(s)$,
$G_\varepsilon(t) \to G(t)$ for every $s, t$. For the $F$ side, we may use
$(b2)$ because $F_{\varepsilon_0}^+$ is integrable. We can therefore pass
to the limit as $\varepsilon \to 0$ and conclude that
\begin{align*}
   \int_{\Omega_1} F_\varepsilon(s) \, d\mu_1(s)
 &= \int_{\Omega_1}
    - \log \Bigl( \int_{\R^n} e^{ - f_s(x) - \varepsilon |x|^2} \, dx \Bigr)
     \, d\mu_1(s)
 \\
 &\to
   \int_{\Omega_1}
    - \log \Bigl( \int_{\R^n} e^{ - f_s } \Bigr)
     \, d\mu_1(s).
\end{align*}
For the analogous expressions with $g_t$, we use $(b1)$ of
Lemma~\ref{LimitsLemma}.

 Having reduced the problem to the Gaussian measure 
$\gamma = \gamma_{n, \tau}$, our goal is now to show how to relax the
continuity assumption from Proposition~\ref{GaussianPropo}. To this end, we
shall introduce classical continuous approximations obtained by
inf-convolution with large multiples of $x \mapsto |x|^2$. For working with
these approximations, we shall need the existence of a selection 
$\alpha_0 \in L^2(X_1, \R^n)$ such that
$\int_{\Omega_1} f_s(\alpha_0(s)) \, d\mu_1(s) < +\infty$, which is granted
by Lemma~\ref{L2selection} since we know by~\eqref{dominated} that
$s \mapsto \log^- \Bigl( \int_{\R^n} e^{-f_s(x)} \, d\gamma(x) \Bigr)$ is
$\mu_1$-integrable when $2 \tau > 1 / \varepsilon_0$.  

 We fix $k > 0$ and define $f_{s, k}$, $g_{t, k}$ by the inf-convolution of 
$f_s, g_t$ with $h_k(x) = k |x|^2$,
\begin{equation} \label{InfConvolu}
 f_{s, k}(x) = \inf_{u \in \R^n} \bigl( f_s(x + u) + k |u|^2 \bigr),
 \quad
 g_{t, k}(x) = \inf_{u \in \R^n} \bigl( g_t(x + u) + k |u|^2 \bigr).
\end{equation}
Clearly, $f_{s, k} \le f_s$, $g_{t, k} \le g_t$, and $f_{s, k}$, $g_{t, k}$
are continuous functions on $\R^n$. By Lemma~\ref{L2selection}, we may find
negligible Borel sets $N_1 \in \Sigma_1$, $N_2 \in \Sigma_2$ such that 
$(s, x) \mapsto f_{s, k}(x)$, $(t, x) \mapsto g_{t, k}(x)$ are Borel
functions on $(\Omega_1 \setminus N_1) \times \R^n$ and 
$(\Omega_2 \setminus N_2) \times \R^n$. We have that
\[
     0
 \le f_{s, k}(x) 
 \le f_s(\alpha_0(s)) + k |x - \alpha_0(s)|^2,
 \quad
 x \in \R^n,
\]
fitting the exponential bound~\eqref{ExponeBound} needed to apply 
Proposition~\ref{GaussianPropo}. Let $\alpha$ be in $L^2(X_1, \R^n)$. For
any fixed $\varepsilon > 0$, we may find a measurable selection $u(s)$
(Lemma~\ref{L2selection}) such that
\[
     f_s \bigl( \alpha(s) + u(s) \bigr) + k |u(s)|^2 - \varepsilon
  <  f_{s, k}( \alpha(s) )
 \le f_s(\alpha_0(s)) + k |\alpha(s) - \alpha_0(s)|^2.
\]
Since $f_s \ge 0$, this shows that $u \in L^2(X_1, \R^n)$. We can write
\[
     \int_{\Omega_2} g_{t, k}( (A \alpha)(t) ) \, d\mu_2(t)
 \le \int_{\Omega_2}
      \bigl[ g_t \bigl( (A \alpha)(t) + (A u)(t) \bigr)
                       + k |(A u)(t)|^2 \bigr]
       \, d\mu_2(t),
\]
which is bounded, using the basic assumption~\eqref{condPL} and the norm
condition $\|A\| \le 1$, by
\[
     \int_{\Omega_1}
      \bigl[ f_s \bigl( \alpha(s) + u(s) \bigr) + k |u(s)|^2 \bigr]
       \, d\mu_1(s)
 \le \int_{\Omega_1} f_{s, k}(\alpha(s)) \, d\mu_1(s)
      + \varepsilon \mu_1(\Omega_1).
\]
Hence, $f_{s, k}$ and $g_{t, k}$ also satisfy the basic assumption. By
Proposition~\ref{GaussianPropo}, we conclude that
\[
     \int_{\Omega_2} \!
      - \log \Bigl( \int_{\R^n} e^{- g_{t, k}(x)} \, d\gamma(x) \Bigr)
         \, d\mu_2(t)
 \le \int_{\Omega_1} \!
      - \log \Bigl( \int_{\R^n} e^{- f_{s, k}(x)} \, d\gamma(x) \Bigr)
       \, d\mu_1(s).
\]
When $k$ tends to infinity, $f_{s, k}$ increases to a \lsc\ function
$\widetilde f_s \le f_s$, and $g_{t, k}(x)$ increases to $g_t(x)$ because
$x \mapsto g_t(x)$ is \lsc\ on $\R^n$. We apply again
Lemma~\ref{LimitsLemma}, this time with $\nu = \gamma$, defining $F_k(s)$, 
$\widetilde F(s) \le F(s)$, $G_k(t)$, $G(t)$ from $f_{s, k}$, 
$\widetilde f_s, f_s$, $g_{t, k}, g_t$ as in~\eqref{FunctionH}. Since the
functions are $\ge 0$, $e^{-f_{s, k}(x)}$, $e^{-g_{t, k}(x)}$ are bounded
by $1$, we have $F_k(s), G_k(t) \ge 0$ because $\gamma$ is a probability
measure. Thus $F_0(s), G_0(t) > - \infty$ and $F_k(s) \to \widetilde F(s)$,
$G_k(t) \to G(t)$ by~$(a)$. The conclusion follows, by $(a1)$ for $F$, and
by $(a2)$ for $G$ since $G_0^- = 0$.
\end{proof}

\begin{rmk} \label{GaussianRmk}
Assume that $f_s$, $g_t$ are continuous $\ge 0$ on $\R^n$. Let $D$ be a
countable dense subset of $\R^n$. Since $f_s$ and $g_t$ are continuous on
$\R^n$, we may define the inf-convolution by
\[
 f_{s, k}(x) = \inf_{u \in D} \bigl( f_s(x + u) + k |u|^2 \bigr),
 \quad
 g_{t, k}(x) = \inf_{u \in D} \bigl( g_t(x + u) + k |u|^2 \bigr).
\]
It is now clear that $(s, x) \mapsto f_{s, k}(x)$, 
$(t, x) \mapsto g_{t, k}(x)$ are
$\Sigma_i \otimes \mathcal{B}_{\R^n}$-measurable, as countable infima of
measurable functions. Similarly, the possibility of selecting $u(s)$ is now
evident.
\end{rmk}

\section{Other formulations and consequences}

We start with a simple, natural situation where the assumptions of
Theorem~\ref{genPL1} are satisfied. Assume that $\Omega$ is a Polish space,
$\Sigma$ its Borel $\sigma$-algebra and $\mu$ a probability measure on
$(\Omega, \Sigma)$. Take $\Sigma_1 = \Sigma$, let $\Sigma_2 \subset \Sigma$
be a sub-$\sigma$-algebra of $\Sigma$, and let $\mu_1 = \mu_2 = \mu$. Then
the conditional expectation
\[
 A = \E \bigl[ \; \cdot \; | \; \Sigma_2 \bigr] :
       L^2( \Sigma, \R^n) \to L^2(\Sigma_2, \R^n)
       \subset L^2( \Sigma, \R^n)
\]
satisfies the norm condition~\NormCond\ and the ``constant functions
condition"~\ConstCond\ from \SettA. An already
interesting case is when we take $\Sigma_2$ to be trivial, 
$\Sigma_2 = \{\emptyset, \Omega\}$, in which case $A$ is simply the
$\mu$-mean,
\[
 A \alpha = \int_\Omega \alpha(s) \, d\mu(s),
\]
and the space $\Omega_2$ can be then considered as being a one-point space,
say $\Omega_2 = \{0\}$. In this case, the family $\{ g_t \}$ consists of a
single function $g$ and Theorem~\ref{genPL1} reads as:

\begin{cor}\label{genPL2}
Let $\Omega$ be a Polish space, $\Sigma$ its Borel $\sigma$-algebra, $\mu$ a
probability measure on $(\Omega, \Sigma)$. Suppose that we are in
\SettA\ with 
$(\Omega_1, \Sigma_1, \mu_1) = (\Omega, \Sigma, \mu)$ and 
$\Omega_2 = \{0\}$. Let $\{ f_s \}_{s \in \Omega}$ satisfy the assumptions
of Theorem~\ref{genPL1}, and let $g$ be a bounded below \lsc\ function on
$\R^n$. If we have for every $\alpha \in L^2(\Omega, \Sigma, \mu, \R^n)$
that
\[
     g \left( \int_\Omega \alpha(s) \, d\mu(s) \right) 
 \le \int_\Omega f_s \bigl( \alpha(s) \bigr) \, d\mu(s),
\]
then
\[
     - \log \left( \int_{\R^n} e^{-g} \right)
 \le \int_\Omega - \log \left( \int_{\R^n} e^{-f_s} \right) \, d\mu(s).
\]
\end{cor} 

 The Pr\'ekopa--Leindler inequality follows by taking $\Omega$ to be a two
point probability space, for instance $\Omega = \{0, 1\}$, and taking $\mu$
of the form $\mu = (1-t) \delta_0 + t \delta_1$ for some $t \in [0, 1]$. 
If we replace the two point space ({\it i.e.}, Bernoulli variables) by the
unit circle $S^1 = \{ e^{i \theta} \; ; \ \theta \in \R \}$
({\it i.e.}, Steinhaus variables), then we obtain:

\begin{cor}
Under the assumptions of Corollary~\ref{genPL2} with $\Omega = S^1$ and
$\mu = d \theta / (2 \pi)$, let $\{ f_\xi \}_{\xi \in S^1}$ be non-negative
(or properly bounded from below as in Remark~\ref{rem:restrictions}) Borel
functions on $\R^n$ and let $g$ be a bounded from below \lsc\ function on
$\R^n$. If for every $\alpha \in L^2(S^1, \R^n)$, we have 
\[
     g \left(
        \int_0^{2\pi} \alpha(e^{i\theta}) \, \frac{d\theta}{2\pi} 
       \right)
 \le \int_0^{2\pi} f_{e^{i\theta}} \bigl( \alpha(e^{i\theta}) \bigr) \,
      \frac {d\theta} {2\pi},
\]
then it follows that 
\[
     - \log \left( \int_{\R^n} e^{-g} \right)
 \le \int_0^{2\pi} 
      - \log \left( \int_{\R^n} e^{-f_{e^{i\theta}}} \right) \,
       \frac{d\theta}{2\pi}.
\]
\end{cor}

 A consequence of the previous Corollary is one of the Berndtsson's
plurisubharmonic extensions of the Pr\'ekopa theorem, the relatively easy
``tube" case. 

\begin{cor}[Berndtsson~\cite{Berndtsson:1998p255}]
Let $\varphi : \C \times \C^n \to \R$ be plurisubharmonic on $\C^{n+1}$ and
such that, for every $z \in \C$ and $w \in \C^n$, we have
$\varphi(z, w ) = \varphi(z, \Re w)$. Then, the function
\[
 \psi(z) = -\log \left( \int_{\R^n} e^{-\varphi(z, x)} \, dx \right)
\]
is subharmonic on $\C$.
\end{cor}

\begin{proof}
We will check that $\psi$ is subharmonic at $z = 0$, say. We can assume that
$\varphi$ is bounded below for all $z$ close to $0$. We want to prove that 
$\psi(0) \le \int_0^{2\pi} \psi(r e^{i\theta}) \, \frac {d\theta} {2\pi}$
for any fixed $r > 0$ small enough. We take $r = 1$ to simplify notations.
Take $g(x) = \varphi(0, x)$ and 
$f_{e^{i \theta}}(x) := \varphi(e^{i \theta}, x )$ for 
$x \in \R^n \subset \C^n$. To conclude, it suffices to check that these
functions satisfy the hypothesis in the previous corollary. Let 
$\alpha \in L^2(S^1, \R^n)$ and let $\widetilde \alpha$ be the harmonic
extension of $\alpha$ to the unit disc $\mathbb D$. In particular
$\widetilde \alpha(0)
 = \int_0^{2\pi} \alpha(e^{i\theta}) \, \frac {d\theta} {2\pi}$. We can
write this $\widetilde \alpha$ as the real part $\Re H$ of an holomorphic
function $H : \mathbb D \to \C^n$ such that $H(0) = \widetilde \alpha(0)$.
We conclude by noticing that $z \mapsto \varphi(z, H(z))$ is subharmonic on
$\mathbb D$, and using 
$\varphi(z, H(z)) = \varphi(z, \widetilde \alpha(z))$ we obtain that 
\begin{align*}
     g \left( \int_0^{2 \pi} \alpha(e^{i \theta}) 
               \, \frac {d\theta} {2 \pi} 
       \right)
  &=  \varphi(0, H(0))
 \\
 &\le \int_0^{2 \pi} 
      \varphi(e^{i \theta}, H(e^{i \theta})) \, \frac{d\theta}{2 \pi}
  =  \int_0^{2 \pi} 
      f_{e^{i \theta}}(\alpha(e^{i \theta})) \, \frac{d\theta}{2 \pi}.
\end{align*}
\end{proof}

 Actually, it is easily seen and certainly known that the previous
Corollary~\ref{genPL2} can be deduced from the Pr\'ekopa--Leindler
inequality, by discretizing $\mu$ and by a simple induction procedure.
Therefore, Corollary~\ref{genPL2} and its consequences only serves as
motivation for the abstract setting above, but do not bring new
information. 

 The situation is different with other instances of Theorem~\ref{genPL1}
such as Proposition~\ref{propM} and the following general result. It seems
difficult to guess the existence of such inequalities without having
Borell's proof in mind.

\begin{theo} \label{GtAndFi}
Let $\mu$ be a Borel probability measure on $[0, 1]$ and set 
$m := \int t \, d\mu(t) \in [0, 1]$. Under \SettA\ with
$\Omega_1 = \{0, 1\}$ and $\Omega_2 = [0, 1]$, $\mu_2 = \mu$, let $f_0$,
$f_1$ and $\{ g_t \}_{t \in [0, 1]}$ be bounded from below (as in
Remark~\ref{rem:restrictions}) Borel functions from $\R^n$ to 
$\R \cup \{+\infty\}$. Assume $g_t$ \lsc\ for every $t$ in $\Omega_2$. If
for every $x_0, x_1 \in \R^n$ we have
\[
     \int_0^1 g_t \bigl( (1-t) x_0 + t x_1 \bigr) \, d\mu(t) 
 \le (1 - m) f_0(x_0) + m f_1(x_1),
\]
then it follows that
\[
     -  \int_0^1 \log \Bigl( \int_{\R^n} e^{-g_t} \Bigr) \, d\mu(t) 
 \le - (1 - m) \log \Bigl( \int_{\R^n} e^{-f_0} \Bigr)
     - m \log \Bigl( \int_{\R^n} e^{-f_1} \Bigr).
\]
\end{theo}

 The deduction of this result from Theorem~\ref{genPL1} is as follows. Take
\[
   X_1
 = (\Omega_1, \mu_1)
 = \bigl( \{0, 1\}, (1-m) \delta_0 + m \delta_1 \bigr), \quad 
   X_2
 = (\Omega_2, \mu_2)
 = \bigl( [0, 1], \mu \bigr)
\]
and for $(v_0, v_1) \in \R^n \times \R^n \simeq L^2(X_1, \R^n)$ define 
$A(v_0, v_1) \in L^2(X_2, \R^n)$ by
\[
 A(v_0, v_1) : t \mapsto (1 - t) v_0 + t v_1.
\] 
Then the constant functions condition~\ConstCond\ is satisfied, as 
$A(v, v) \equiv v$, and by the convexity of the square of the Euclidean
norm on $\R^n$, letting $H_i = L^2(X_i, \R^n)$, $i = 1, 2$, we have 
\begin{align*}
     \|A(v_0, v_1) \|^2_{H_2}
  &=  \int_0^1 |(1 - t) v_0 + t v_1|^2 \, d\mu(t) 
 \\
 &\le (1 - m) |v_0|^2 + m |v_1|^2 
  =  \|(v_0, v_1)\|^2_{H_1}.
\end{align*}

\begin{rmk} \label{RmkGtAndFi}
It is possible to prove Theorem~\ref{GtAndFi} without assuming $g_t$ \lsc\
and $f_i$ bounded from below, see Claim~\ref{OneMore} for a sketch of proof.
\end{rmk}

 As a particular case of the previous proposition, if we take only one
function $g_t \equiv g : \R^n \to \R$, we have that when $\mu$ is a
probability measure on $[0, 1]$ with barycenter $m$, and
\begin{equation}
\label{condPLmu}
 \forall x_0, x_1 \in \R^n, \ 
     \int_0^1 g \bigl( (1-t) x_0 + t x_1 \bigr) \, d\mu(t) 
 \le (1-m) f_0(x_0) + m f_1(x_1),
\end{equation}
then
\[
     \int_{\R^n} e^{-g}
 \ge \Bigl( \int_{\R^n} e^{-f_0} \Bigr)^{1-m} 
     \Bigl( \int_{\R^n} e^{-f_1} \Bigr)^{m}.
\]
The conclusion is independent of $\mu$, so given the functions $f_0$, 
$f_1, g$, one can try to find an optimal $\mu$ for which the
condition~\eqref{condPLmu} holds. The classical Pr\'ekopa--Leindler
inequality cor\-res\-p\-onds to the choice $\mu = \delta_m$, and this is
indeed the optimal choice when $g$ is convex, as seen by using Jensen's
inequality in~\eqref{condPLmu}. Does the above result really improve on the
Pr\'ekopa--Leindler inequality when $g$ is non-convex?

 Another particular case of the previous proposition (in the case of convex
combination of two Dirac measures) is the following extension of
Proposition~\ref{propM}.

\begin{prop}\label{propM-gen}
Fix $s, t, r \in [0, 1]$ and set $m := (1-r) s + r t \in [0, 1]$. Let 
$f_0, f_1$, $g_0, g_1$ be four Borel functions from $\R^n$ to
$\R \cup \{+\infty\}$ such that for every $x_0, x_1 \in \R^n$,
\begin{equation*}
     (1-r) g_0 \bigl( (1-s) x_0  + sx_1 \bigr)
      + r g_1 \bigl( (1-t) x_0  + t x_1 \bigr) 
 \le (1-m) f_0(x_0) + m f_1(x_1).
\end{equation*}
Then
\begin{align*}
     \left( \int_{\R^n} e^{-g_0(x)} \, dx \right)^{1-r}  
      & \left(\int_{\R^n} e^{-g_1(x)} \, dx \right)^r
 \\
 & \ge \left(\int_{\R^n} e^{-f_0(x)} \, dx \right)^{1-m} 
     \left(\int_{\R^n} e^{-f_1(x)} \, dx \right)^m.
\end{align*}
\end{prop}

\begin{proof}
Although the result is a particular case of Theorem~\ref{GtAndFi} (with
$\mu = (1-r) \delta_{s} + r \delta_{t}$), it is better to go back to
Theorem~\ref{discretePL1} since in the case of finitely many functions
there are no technical conditions. We take 
$\Omega_1 = \{\{0, 1\}, (1-m)\delta_0 + m\delta_1\}$ and 
$\Omega_2 = \{\{0, 1\}, (1-r)\delta_0 + r\delta_1\}$, and the linear
mapping $A : H_1 \to H_2$ with $H_i := L^2(\Omega_i, \R^n)$ defined by 
\[
    A(x_0, x_1) 
 := ((1 - s) x_0 + s x_1, (1-t) x_0 + t x_1), \quad x_0, x_1 \in \R^n.
\] 
We have $A(v, v) = (v, v)$ for every $v \in \R^n$ and we note that for 
$x_0, x_1 \in \R^n$,
\begin{align}
   (1 & - r) |(1-s) x_0 + s x_1|^2 + r |(1-t) x_0 + t x_1|^2
 \label{moremoresquares}
 \\
    & + ((1-r) s (1-s) + r t (1-t)) |x_0 - x_1|^2 
 = (1-m) |x_0|^2 + m |x_1|^2
 \notag
\end{align}
and so in particular 
\[ 
     (1-r) |(1-s) x_0 + s x_1|^2 + r |(1-t) x_0 + t x_1|^2  
 \le (1-m) |x_0|^2 + m |x_1|^2
\]
which exactly means that $\|A\| \le 1$.
\end{proof}

 Let us mention that Proposition~\ref{propM} (and also the above results)
do not give anything more than the Brunn--Minkowski inequality when applied
to the case of indicator functions, {\it i.e.}, to 
$e^{-f_0} = \mathbf 1_{A_0}$, $e^{-f_1} = \mathbf 1_{A_1}$, 
$e^{-g_0} = \mathbf 1_{2 A_0 / 3 + A_1 / 3}$, 
$e^{-g_1} = \mathbf 1_{A_0 / 3 + 2 A_1 / 3}$. Indeed, by Brunn--Minkowski,
one has that
\[
 |2 A_0 / 3 + A_1 / 3| \ge |A_0|^{2/3} |A_1|^{1/3},
 \quad
 |A_0 / 3 + 2 A_1 / 3| \ge |A_0|^{1/3} |A_1|^{2/3},
\]
and in this case, the result of Proposition~\ref{propM} is obtained by
taking the product of these two inequalities. It seems that the extra
information comes from applying the results to ``true" functions. This is
consistent with the fact that we do not know how to reach our functional
inequalities using other classical proofs of the Pr\'ekopa--Leindler
inequalities. For instance, it is not clear that Proposition~\ref{propM}
can be proved using the mass transportation argument of
McCann~\cite{mccann}; precisely, this transportation argument uses some
form of ``localization inside the integral" amounting to reducing the
problem to sets (ellipsoids, actually) and eventually matrices.

\begin{expl}[Gaussian self-improvement and generalized $\tau$-property]
Let $\alpha \in (0, 1)$. We shall comment on Proposition~\ref{propM-gen} in
the case $s = \alpha$, $t = 1-\alpha = 1-s$ and $r = 1/2$ (and so 
$m = 1/2$).  The goal is to get improved Gaussian inequalities by exploiting
identity~\eqref{moremoresquares} instead of an inequality.  Let $f_0, f_1$
be real Borel functions on $\R^n$.

 We start first with the case of the Pr\'ekopa-Leindler inequality and
consider the following variant $g$ of inf-convolution, defined for every 
$x \in \R^n$ by
\begin{align*}
   g(x) 
 = \inf \{ (1 - \alpha) f_0(x_0) + \alpha f_1(x_1) 
            + \alpha (1 - & \alpha) |x_0 - x_1|^2 / 2 :
 \\
         & x = (1 - \alpha) x_0 + \alpha x_1 \}.
\end{align*}
We assume that $g(x) > -\infty$. Using the following particular case
of~\eqref{moremoresquares},
\begin{equation} \label{addingsquares}
   |(1 - \alpha) x_0 + \alpha x_1|^2
    + \alpha (1 - \alpha) |x_0 - x_1|^2
 = (1 - \alpha) |x_0|^2 + \alpha |x_1|^2,
\end{equation}
it follows that
$     g(x) + |x|^2 / 2
  \le (1 - \alpha) \bigl( f_0(x_0) + |x_0|^2 / 2 \bigr)
       + \alpha \bigl( f_1(x_1) + |x_1|^2 / 2 \bigr)$ 
whenever $x = (1 - \alpha) x_0 + \alpha x_1$, and we obtain by
Pr\'ekopa--Leindler that
\[
     \int_{\R^n} e^{- g} \, d\gamma_n
 \ge \Bigl( \int_{\R^n} e^{- f_0} \, d\gamma_n \Bigr)^{1 - \alpha}
     \Bigl( \int_{\R^n} e^{- f_1} \, d\gamma_n \Bigr)^{\alpha},
\]
where $\gamma_n$ is the standard Gaussian measure $\gamma_{n, 1}$ on
$\R^n$. This infimal convolution inequality ensures the $\tau$-property
from~\cite{maurey91} for the Gaussian measure.

 For $\alpha, \beta, \lambda \in (0, 1)$, we may generalize $g$ as
\begin{align*}
   g_{\alpha, \beta, \lambda} (x) 
 = \inf \{ (1 - \beta) f_0(x_0) + \beta f_1(x_1) 
            + \lambda \alpha (1 - & \alpha) |x_0 - x_1|^2 :
 \\
         & x = (1 - \alpha) x_0 + \alpha x_1 \},
\end{align*}
the former $g$ being equal to $g_{\alpha, \alpha, 1/2}$. The
Pr\'ekopa--Leindler inequality does not seem to apply to values of the
parameters other than triples of the form $(\alpha, \alpha, 1/2)$. 
Combining
\[
     g_{\alpha, \beta, \lambda} (x_\alpha)
 \le (1 - \beta) f_0(x_0) + \beta f_1(x_1)
      + \lambda \alpha (1 - \alpha) |x_0 - x_1|^2
\]
and the corresponding inequality for
$g_{1 - \alpha, 1 - \beta, 1 - \lambda} (x_{1 - \alpha})$,
we get, with $g_0 = g_{\alpha, \beta, \lambda}$ and 
$g_1 = g_{1 - \alpha, 1 - \beta, 1 - \lambda}$, that 
\[
     g_0 (x_\alpha) + g_1 (x_{1 - \alpha})
 \le f_0(x_0) + f_1(x_1) + \alpha (1 - \alpha) |x_0 - x_1|^2.
\]
The identity~\eqref{moremoresquares} in the case $s = \alpha$, 
$t = 1 - \alpha$ and $r = 1/2$ (thus $m = 1/2$) rewrites as
\[
   |(1 - \alpha) x_0 + \alpha x_1|^2
   + |\alpha x_0 + (1 - \alpha) x_1|^2
   + 2 \alpha (1 - \alpha) |x_0 - x_1|^2 = |x_0|^2 + |x_1|^2 
\]
and so we find 
\[
     g_0 (x_\alpha) + |x_\alpha|^2 / 2
      + g_1 (x_{1 - \alpha}) + |x_{1 - \alpha}|^2 / 2
 \le f_0(x_0) + |x_0|^2 / 2 + f_1(x_1) + |x_1|^2 / 2.
\]
Therefore, by Proposition~\ref{propM-gen} with $s = \alpha$, 
$t = 1 - \alpha$ and $r = 1/2 = m$, we arrive at the following generalized
infimal convolution inequality:
\[
     \Bigl(
      \int_{\R^n} e^{-g_{\alpha, \beta, \lambda}} \, d\gamma_n 
      \Bigr)
     \Bigl(
      \int_{\R^n} e^{-g_{1 - \alpha, 1 - \beta, 1 - \lambda}} \, d\gamma_n 
     \Bigr)
 \ge \Bigl( \int_{\R^n} e^{-f_0} \, d\gamma_n \Bigr)
     \Bigl( \int_{\R^n} e^{-f_1} \, d\gamma_n \Bigr).
\]
\end{expl}

\begin{expl}[Exotic situation]
All our examples so far of linear maps $A$ are of ``convex type", meaning
that
\[
     \int_{\Omega_2} \varphi \bigl( A(\alpha)(t) \bigr) \, d\mu_2(t)
 \le \int_{\Omega_1} \varphi( \alpha(s) ) \, d\mu_1(s)
\]
for every convex function $\varphi$ on $\R^n$, while the norm
condition~\NormCond\ on $A$ ensures this property for 
$\varphi(x) = k |x|^2$ only, $k \ge 0$. Here is a ``non convex" example,
with $\Omega_1 = \Omega_2 = \Omega = \{0, 1\}$ and 
$\mu_1 = \mu_2 = \delta_0 + \delta_1$. The space $L^2(\Omega, \R^n)$ is
equal to $\R^n \times \R^n$, the matrix $A$ can be represented by blocks of
size $n \times n$,
\[
 A = \left( \begin{matrix}
       A_1 & B_1 \cr
       C_1 & D_1 \cr
   \end{matrix} \right).
\]
By the constant functions condition~\ConstCond, the image of the constant
function equal to $x \in \R^n$ is 
$(A_1 x + B_1 x, C_1 x + D_1 x) = (x, x)$, which is equivalent to 
$A_1 + B_1 = C_1 + D_1 = I_n$, the identity matrix. One can thus write
\[
 A = \left( \begin{matrix}
       I - B &   B   \cr
         C   & I - C \cr
   \end{matrix} \right).
\]
Since $\|A\| \le 1$, each block must have norm $\le 1$, and this implies
that the diagonal coefficients of $B, C$ are in $[0, 1]$. The condition
$\|A\| \le 1$ means that $A^* A \le I_{2 n}$, which translates to
\[
   \left( \begin{matrix}
     \hphantom{-} B^* B + C^* C  &            - B^* B - C^* C \cr
                - B^* B - C^* C  & \hphantom{-} B^* B + C^* C \cr
   \end{matrix} \right)
   \le
   \left( \begin{matrix}
       \hphantom{-} B + B^*  &            - B - C^* \cr
                  - B^* - C  & \hphantom{-} C + C^* \cr
   \end{matrix} \right).
\]
In the simpler case $C = B$, this amounts to $2 B^* B \le B + B^*$, or 
$\|B x\|^2 \le Bx \cdot x$, for every $x \in \R^n$. For an elementary
explicit example, take $n = 2$ and
\[
   B
 = \left( \begin{matrix}
    \hphantom{-}   b     &  \varepsilon \cr
         - \varepsilon   &      b       \cr
   \end{matrix} \right)
\]
with $b \in (0, 1)$ and $b^2 + \varepsilon^2 \le b$. We see that 
$B = b I_2 - \varepsilon R$, with $R$ the rotation by $\pi / 2$ in the
plane $\R^2$. Then, with $b = 1/3$ and $\varepsilon = \sqrt 2 / 3$, if we
know that
\begin{align*}
     & g_0 \Bigl( \frac 2 3 x_0 + \frac 1 3 x_1
               - \frac {\sqrt 2} 3 R(x_1 - x_0) \Bigr)
      + g_1 \Bigl( \frac 1 3 x_0 + \frac 2 3 x_1 
                   + \frac {\sqrt 2} 3 R(x_1 - x_0) \Bigr) 
 \\
 \le & f_0(x_0) + f_1(x_1)
\end{align*}
for all $x_0, x_1 \in \R^n$, we get the same conclusion as that of
Proposition~\ref{propM}.
\end{expl}

\begin{rmk} \label{noexotic}
Following the previous construction, it is natural to ask if we can
construct an ``exotic" Pr\'ekopa--Leindler situation. The answer is no. 
Let $\alpha \in [0, 1]$. The only $n \times n$ (real) matrix $B$ such that
$|(I_n - B) x + B y|^2 \le (1 - \alpha) |x|^2 + \alpha |y|^2$ for all 
$x, y \in \R^n$ is $B = \alpha I_n$ (try $x = u + t \alpha v$,
$y = u - t (1 - \alpha) v$ and $t \to 0$). In other words, there is no
``exotic example" in the Pr\'ekopa--Leindler case. More generally, if 
$B_1, \ldots, B_k$ are $n \times n$ matrices such that 
$\sum_{j=1}^k B_j = I_n$ and 
$|\sum_{j=1}^k B_j x_j|^2 \le \sum_{j=1}^k \alpha_j |x_j|^2$, with
$\alpha_j \ge 0$ and $\sum_{j=1}^k \alpha_j = 1$, then 
$B_j = \alpha_j I_n$ for $j = 1, \ldots, k$.
\end{rmk}

\section{ \hskip-6pt Generalized Brascamp--Lieb and reverse
 Brascamp--Lieb inequalities} \label{GeneBL}
By slightly modifying \SettA, it is possible to recover and
extend, almost for free, the Brascamp--Lieb inequalities and their reverse
forms. Let us mention that it has been known for some time that the
Brascamp--Lieb inequalities can be recovered using Borell's technique (see
e.g.~\cite{lehec}).

 If the functions $\{ f_s \}$ are defined on $\R^m$ and $\{ g_t \}$ on
$\R^n$, with $m \ne n$, then the linear operator $A$ of \SettA\
should now act from $L^2(X_1, \R^m)$ to $L^2(X_2, \R^n)$ and the constant
functions condition~\ConstCond\ has to be revised. We shall do it by using
projections from the larger space, $\R^m$ or $\R^n$, onto the smaller. To
be precise, our \emph{projections} are adjoint to isometries from the
smaller space into the larger. For example, if $n < m$ and if $T$ is an
isometry from $\R^n$ into $\R^m$, its adjoint $Q = T^*$ is a mapping from
$\R^m$ onto $\R^n$ such that $Q Q^* = \id_{\R^n}$, and $Q^* Q$ is the
orthogonal projection of $\R^m$ onto the range of $T$. Then, for 
$v \in \R^m$, we can compare the image $A(s \mapsto v)$ of the constant
function $\Omega_1 \ni s \mapsto v$ with the constant function 
$\Omega_2 \ni t \mapsto Q v$. Actually, the new setting will be notably
more complicated, introducing a \emph{family} of projections $Q(t)$, 
$t \in \Omega_2$, and comparing the image $A(s \mapsto v)$ with the
function $t \mapsto Q(t) v$. We can also view the new setting as giving a
measurable family of $n$-dimensional subspaces $X(t)$ of $\R^m$,
parameterized by $T(t) : \R^n \to X(t)$, and $Q(t) = T(t)^*$ being the
composition of the orthogonal projection $\pi(t)$ of $\R^m$ onto $X(t)$
with the inverse map of $T(t)$.
 
 Except for what concerns the linear mapping $A$, the modifications are
straightforward, and will be just indicated without rewriting completely
the modified assumption.

\begin{settingB}
The definition of the measure spaces $X_1 = (\Omega_1, \Sigma_1, \mu_1)$,
$X_2 = (\Omega_2, \Sigma_2, \mu_2)$ is as in \SettA, in particular, $\mu_1$
and $\mu_2$ are finite measures. Two integers $m, n \ge 1$ are given, and
the linear operator $A$ acts now as
\[
 A : L^2(X_1, \R^m) \to L^2(X_2, \R^n)
\]
with $\|A\| \le 1$, where the norms are computed with respect to the
Euclidean norm $|\cdot |$ on $\R^m$, $\R^n$ and to the measures $\mu_1$ and
$\mu_2$, respectively. The constant functions condition is modified as
follows:

\begin{itemize}
\item[\ConstCondA] 
if $m \le n$, there exists a $\Sigma_1$-measurable family 
$\Omega_1 \ni s \mapsto P(s)$ of projections $P(s)$ from $\R^n$ onto $\R^m$
such that, for every vector $w_0 \in \R^n$, 
$A \bigl(s \mapsto P(s) w_0 \bigr)(t) = w_0$ for $\mu_2$-almost every 
$t \in \Omega_2$,
\item[\ConstCondB] if $m \ge n$, there exists a $\Sigma_2$-measurable
family $\Omega_2 \ni t \mapsto Q(t)$ of projections $Q(t)$ from $\R^m$ onto
$\R^n$ such that, for every vector $v_0 \in \R^m$, 
$A ( s \mapsto v_0 ) (t) = Q(t) v_0$ for $\mu_2$-almost every 
$t \in \Omega_2$.
\end{itemize}
Note that formally, \ConstCondB\ is the adjoint situation to~\ConstCondA.
Observe that~\ConstCondA\ implies that every $w_0 \in \R^n$ can be
reconstructed from the family of projections $s \mapsto P(s) w_0 \in \R^m$.

 For the conditions on $\{f_s\}$, $\{g_t\}$, we need only replace $\R^n$ by
$\R^m$ for what concerns $f_s$ in the measurability condition~\MeasuCond\
and in the semi-integrability condition~\SemiInt.

\end{settingB}

 We arrive at the following extension of Theorem~\ref{genPL1}.

\begin{theo}\label{genPL}
Under \SettB, we make the additional assumptions that the
functions $f_s, g_t$ are non-negative with $g_t$ lower semicontinuous on
$\R^n$, that for some $\varepsilon_0 > 0$ we have
\begin{equation} \label{domina2}
 \int_{\Omega_1} \log^- 
  \Bigl( 
   \int_{\R^m} \exp \bigl( - f_s(x) - \varepsilon_0 |x|^2 \bigr) \, dx 
  \Bigr) 
   \, d\mu_1(s) < +\infty,
\end{equation}
and that the measures $\mu_1, \mu_2$ are such that
$m \cdot \mu_1(\Omega_1) = n \cdot \mu_2(\Omega_2) < +\infty$.

 If for every $\alpha \in L^2(X_1, \R^m)$, we have
\begin{equation}\label{basic2}
     \int_{\Omega_2} g_t \bigl( (A \alpha) (t) \bigr) \, d\mu_2(t) 
 \le \int_{\Omega_1} f_s (\alpha(s)) \, d\mu_1(s),
\end{equation}
then
\[
     \int_{\Omega_2}
      -\log \left( \int_{\R^n} e^{-g_t} \right) \, d\mu_2(t) 
 \le \int_{\Omega_1}
      -\log \left( \int_{\R^m} e^{-f_s} \right) \, d\mu_1(s).
\]
\end{theo}

 Not only the argument for deriving Theorem~\ref{genPL} follows the proof
of Theorem~\ref{genPL1}, but in its Gaussian version, Theorem~\ref{genPL}
is already contained in Proposition~\ref{GaussianPropo}. Indeed, after
reducing to the Gaussian case, we further approximate as before the
functions $f_s$ and $g_t$ by inf-convolution, in order to be in a position
to apply the proposition below, just as we did for proving
Theorem~\ref{genPL1}. The assumption 
$m \cdot \mu_1(\Omega_1) = n \cdot \mu_2(\Omega_2)$ is needed to pass to
the limit from the Gaussian case for $P_T$, $T \to +\infty$, in respective
dimensions $m$ and $n$.

 For the Gaussian version below, we will have not much to add, except maybe
for the case~\ConstCondB\ of \SettB. So
Proposition~\ref{GaussianPropo} almost includes all possible geometric
situations around the Pr\'ekopa--Leindler inequality, including these
generalized Brascamp--Lieb inequalities. 

\begin{prop} \label{GaussianPropo2}
Under \SettB, assume that $f_s$, $g_t$ are continuous $\ge 0$ on $\R^m$ and
$\R^n$, respectively (or bounded from below as in
Remark~\ref{rem:restrictions}), and that the functions $f_s$ satisfy the
exponential bound~\eqref{exponebound} on $\R^m$. Then, the basic
assumption~\eqref{basic2} implies that
\begin{align}
     \int_{\Omega_2}
      -\log & \left( \int_{\R^n} e^{-g_t(y)} \, d\gamma_{n, \tau}(y) \right)
       \, d\mu_2(t) 
 \label{conclGL2}
 \\
 &\le \int_{\Omega_1} 
      -\log \left( \int_{\R^m} e^{-f_s(x)} \, d\gamma_{m, \tau}(x) \right)
       \, d\mu_1(s).
 \notag
\end{align}
\end{prop}

 As before, in the case where the measures $\mu_1$ and $\mu_2$ have finite
support, we can remove all the additional assumptions on $f_s$ and $g_t$.

\begin{proof}
We want to apply Proposition~\ref{GaussianPropo} but we have to deal with
the fact that dimensions $m$ and $n$ are now different. We shall do it by
reducing to the case when both dimensions are equal to $\max(m, n)$.
\smallskip

 In case~\ConstCondA, when $m \le n$, we ``extend" $f_s$ to $\R^n$ by
defining $f_{1, s}(y) = f_s(P(s) y)$, $y \in \R^n$, so that 
\begin{equation} \label{ChangeDim}
   \int_{\R^n} e^{- f_{1, s}(y)} \, d\gamma_{n, \tau}(y)
 = \int_{\R^m} e^{- f_s(x)} \, d\gamma_{m, \tau}(x).
\end{equation}
For every $\alpha \in L^2(X_1, \R^n)$, define
$\alpha_1 \in L^2(X_1, \R^m)$ by $\alpha_1(s) = P(s) \alpha(s)$ and set
$A_1(\alpha) = A (\alpha_1)$. Then $A_1$ is linear from $L^2(X_1, \R^n)$
to $L^2(X_2, \R^n)$, and clearly $\|A_1\| \le \|A\| \le 1$. If $w_0$ is a
fixed vector in $\R^n$, then 
$A_1(s \mapsto w_0) = A(s \mapsto P(s) w_0) = w_0$ by~\ConstCondA, so $A_1$
satisfies the constant functions condition~\ConstCond. Next, by the basic
assumption of Theorem~\ref{genPL} applied to~$\alpha_1$, we get that
\[
     \int_{\Omega_2} g_t \bigl( (A_1 \alpha)(t) \bigr) \, d\mu_2(t)
  =  \int_{\Omega_2} g_t \bigl( (A \alpha_1)(t) \bigr) \, d\mu_2(t)
 \le \int_{\Omega_1} f_s \bigl( \alpha_1(s) \bigr) \, d\mu_1(s)
\]
\[
  =  \int_{\Omega_1} f_s \bigl( P(s) \alpha(s) \bigr) \, d\mu_1(s)
  =  \int_{\Omega_1} f_{1, s} ( \alpha(s) ) \, d\mu_1(s).
\]
We see that $A_1$, $\{ f_{1, s} \}$ and $\{ g_t \}$ satisfy the assumptions
of Proposition~\ref{GaussianPropo}, including the basic
assumption~\eqref{condPL} on $\R^n$. The result follows therefore
from~\eqref{ChangeDim} and from the conclusion of
Proposition~\ref{GaussianPropo} for $\{ f_{1, s} \}$ and $\{ g_t \}$.
\smallskip

 In case~\ConstCondB, when $m \ge n$, let $T(t)$ be the isometry from
$\R^n$ into $\R^m$ adjoint to the projection $Q(t)$, so that 
$Q(t) T(t) = \id_{\R^n}$. The function $g_t$ is extended to $\R^m$ by
setting $g_{1, t}(x) = g_t(Q(t) x)$, $x \in \R^m$. For 
$\alpha \in L^2(X_1, \R^m)$, we set $(A_1 \alpha) (t) = T(t) (A \alpha)(t)$
and get that $Q(t) (A_1 \alpha)(t) = (A \alpha)(t)$. We have that 
$\|A_1\| = \|A\| \le 1$. If $\alpha_0 = v_0$ is a constant function from
$\Omega_1$ to $\R^m$, we know by~\ConstCondB\ that 
$(A \alpha_0)(t) = Q(t) v_0$. Then 
$(A_1 \alpha_0)(t) = T(t) (A \alpha_0)(t) = T(t) Q(t) v_0$, not equal to
$v_0$ in general. The constant functions condition~\ConstCond\ is not
satisfied by $A_1$, but still we have
\begin{equation} \label{WasEnough}
   g_{1, t} \bigl( A_1 (v_0 + \alpha)(t) \bigr) 
 = g_{1, t} \bigl( v_0 + (A_1 \alpha)(t) \bigr),
 \quad
 t \in \Omega_2,
\end{equation}
for every $\alpha \in L^2(X_1, \R^m)$. This is because 
$g_{1, t}(x_1) = g_{1, t}(x_2)$ when $Q(t) x_1 = Q(t) x_2$, and because we
have by~\ConstCondB\ that 
\[
   Q(t) A_1 (v_0 + \alpha)(t)
 = A (v_0 + \alpha)(t)
 = Q(t) v_0 + (A \alpha)(t)
 = Q(t) \bigl( v_0 + (A_1 \alpha)(t) \bigr).
\]
The basic assumption~\eqref{condPL} is satisfied for the families
$\{f_s\}$, $\{g_{1,t}\}$ of functions on $\R^m$ and the mapping $A_1$:
since
$g_{1, t} \bigl( (A_1 \alpha)(t) \bigr) = g_t \bigl( (A \alpha)(t) \bigr)$,
we get
\[
     \int_{\Omega_2} 
      g_{1, t} \bigl( (A_1 \alpha)(t) \bigr) \, d\mu_2(t)
  =  \int_{\Omega_2} 
      g_t \bigl( A(\alpha)(t) \bigr) \, d\mu_2(t)
 \le \int_{\Omega_1} f_s ( \alpha(s) ) \, d\mu_1(s).
\]
We stressed during the proof of Proposition~\ref{GaussianPropo} that the
equality~\eqref{WasEnough} is enough to run the argument (see
equation~\eqref{IsEnough} and the lines around it). Finally, as before,
observe that
\[
   \int_{\R^n} e^{- g_t(y)} \, d\gamma_{n, \tau}(y)
 = \int_{\R^m} e^{- g_{1, t}(x)} \, d\gamma_{m, \tau}(x),
\]
and use Proposition~\ref{GaussianPropo} to get the desired conclusion.
\end{proof}

\begin{rmk}
The normalization $\|A\| \le 1$ can be replaced by the following
assumptions. Let $\kappa := \|A\| > 0$ and assume that 
$\kappa^2 \cdot m \cdot \mu_1(\R^m) = n \cdot \mu_2(\R^n)$. If for every
$\alpha$ in $L^2(X_1, \R^m)$, we have
\[
     \int_{\Omega_2} g_t ( (A \alpha)(t) ) \, d\mu_2(t) 
 \le \kappa^2 \int_{\Omega_1} f_s (\alpha(s)) \, d\mu_1(s),
\]
then
\[
     \int_{\Omega_2} 
      - \log \left( \int_{\R^n} e^{-g_t} \right) \, d\mu_2(t) 
 \le \kappa^2 \int_{\Omega_1}
          - \log \left( \int_{\R^m} e^{-f_s} \right) \, d\mu_1(s).
\] 
\rm
The reader will just look at inequality~\eqref{FourEq} and the inequality
before it.
\end{rmk}
\medskip

 We shall examine now particular cases of Theorem~\ref{genPL}. We will see
that it contains both the Brascamp--Lieb and reverse Brascamp--Lieb
inequalities in their geometric form. 

 Let us take unit vectors $u_1, \ldots u_N$ in the Euclidean space $\R^d$
and positive reals $c_1, \ldots, c_N$ that decompose the identity of $\R^d$,
meaning that
\begin{equation} \label{decomp}
   \sum_{i=1}^N c_i \, u_i \otimes u_i 
 = \id_{\R^d}.
\end{equation}
This is equivalent to saying that
$x \cdot x = \sum_{i=1}^N c_i |u_i \cdot x|^2$ for every $x \in \R^d$. If
we consider the one-point space $E_1 = \{0\}$ equipped with the trivial
probability measure $\nu_1$, and the measure space
\[
   (E_2, \nu_2) 
 = \Bigl( \bigl\{ 1, \ldots N \bigr\}, \, 
          \sum_{i=1}^N c_i \delta_i \Bigr),
 \quad
 \nu_2(E_2) = \sum_{i=1}^N c_i = d,
\]
then~\eqref{decomp} is equivalent to saying that the mapping
\begin{equation} \label{isom}
 U : x \in \R^d \mapsto (x \cdot u_1, \ldots, x \cdot u_N)
\end{equation}
is an isometry from $L^2(E_1, \nu_1, \R^d) \simeq \R^d$ into
$L^2(E_2, \nu_2, \R) \simeq \R^N$.

 In order to recover the reverse Brascamp--Lieb inequalities, consider 
\[
 X_1 = (E_2, \nu_2), \ m = 1, 
 \quad
 X_2 = (E_1, \nu_1), \ n = d,
\]
satisfying
\[
 m \mu_1(\Omega_1) = \nu_2(E_2) = d = n \mu_2(\Omega_2).
\]
Define $A : \R^N \simeq L^2(X_1, \R^m) \to \R^d \simeq L^2(X_2, \R^n)$ to
be the linear operator
\[
   A(t_1, \ldots, t_N)
 = \sum_{i=1}^N c_i \, t_i \, u_i,
 \quad 
 t_i \in \R.
\]
Then $A$ is adjoint to the into isometry $U$ of~\eqref{isom} between our
$L^2$ spaces, and $A$ satisfies the condition~\ConstCondA\ of 
\SettB, with the family of projections $P_i v = v \cdot u_i$,
$i \in \Omega_1$, because we have by~\eqref{decomp} that
\[
 A(i \mapsto P_i v) = \sum_{i=1}^N c_i (v \cdot u_i) u_i = v.
\]
The conclusion reads as follows. If for every $t_1, \ldots, t_N \in \R$ we
have
\[
     g \Bigl( \sum_{i=1}^N t_i c_i u_i \Bigr)
 \le \sum_{i=1}^N c_i \, f_i(t_i)
 \quad \text{then} \quad
     \int_{\R^d} e^{-g}
 \ge \prod_{i=1}^N \left( \int_\R e^{-f_i} \right)^{c_i}.
\]
This is the reverse Brascamp--Lieb inequality of
Barthe~\cite{Barthe:1998gu} in its geometric form. Actually, in the same
way, we can see that the formulation above contains the ``continuous"
statement given in~\cite{barthe04} (that can be easily derived by
approximation, anyway).

 In order to recover the classical Brascamp--Lieb inequalities, apply the
theorem with $X_1 = (E_1, \nu_1)$, $m = d$, $X_2 = (E_2, \nu_2)$, $n = 1$
and the linear operator $A : L^2(X_1, \R^m) \simeq 
 \R^d \to L^2(X_2, \R^n) \simeq \R^N$ defined by
$A(x) = (x \cdot u_1, \ldots, x \cdot u_N)$. Then $A = U$, the into
isometry~\eqref{isom} between the corresponding $L^2$ spaces, and it
satisfies the assumption~\ConstCondB\ with the family of projections 
$Q_j(x) = x \cdot u_j$, $j \in \Omega_2 = \{1, \ldots, N\}$. The
conclusion reads as follows: given $f : \R^d \to \R$ and
$g_1, \ldots, g_N : \R \to \R$ such that 
\[
 \forall x \in \R^d, \quad 
     \sum_{j=1}^N c_j g_j (x \cdot u_j)
 \le f(x),
 \quad \text{we have} \quad
     \prod_{j=1}^N \Bigl( \int_\R e^{-g_j} \Bigr)^{c_j}
 \ge \int_{\R^d} e^{-f}.
\]
This is the geometric Brascamp--Lieb inequality, usually stated with the
best possible $f$, which is defined by the equality in place of inequality
above, namely
\[
     \int_{\R^d} 
      \exp \Bigl( - \sum_{j=1}^N c_j g_j (x \cdot u_j) \Bigr) \, dx
 \le \prod_{j=1}^N \Bigl( \int_\R e^{-g_j} \Bigr)^{c_j}.
\]

 It is possible to state (and prove in the same way) a more general form of
Theorem~\ref{genPL} where each function $f_s$ (resp. $g_t$) is defined on a
different Euclidean space $E_s$ (resp. $F_t$), and to deduce from it the
multidimensional geometric Brascamp--Lieb inequalities. In this situation,
we assume that all spaces $E_s, F_t$ are subspaces of a given $\R^\ell$. We
may consider that the space $E_s$ is $\R^{m(s)}$ and that $F_t$ is
$\R^{n(t)}$, with $1 \le m(s), n(t) \le \ell$ and we assume that
\[
   \int_{\Omega_1} m(s) \, d\mu_1(s)
 = \int_{\Omega_2} n(t) \, d\mu_2(t).
\]
The mapping $A$ is defined on the closed subspace $H$ of $L^2(X_1, \R^\ell)$
consisting of those $\alpha$ with $\alpha(s) \in E_s$ for every 
$s \in \Omega_1$. We assume that for every square integrable choice 
$\alpha \in H$, the image $A \alpha \in L^2(X_2, \R^\ell)$ is such that 
$(A \alpha)(t)$ belongs to $F_t$ for every $t \in \Omega_2$. We give
projections $P(s), Q(t)$ from $\R^\ell$ onto $E_s, F_t$ respectively. The
constant functions condition says now that for every $w_0 \in \R^\ell$, we
have that
\[
 A \bigl( s \mapsto P(s) w_0 \bigr) (t) = Q(t) w_0.
\]
The proof mixes the two cases~\ConstCondA\ and~\ConstCondB\ of the proof of
Theorem~\ref{genPL}. We leave this to the reader.

\section{ Technicalities} \label{technical}

\begin{claim} \label{OneMore}
In Theorem~\ref{GtAndFi}, we can remove the assumptions that $g_t$ is \lsc\
and that $f_i \ge 0$ (or bounded from below), $i = 0, 1$.
\end{claim}

\begin{proof}[Sketch of proof]
As in the proof of Theorem~\ref{genPL1}, we begin by replacing the Lebesgue
measure in $\int_{\R^n} e^{-f_i(x)} \, dx$ and 
$\int_{\R^n} e^{-g_t(x)} \, dx$ by a Gaussian probability measure $\gamma$.
The justification is the same here for the $g_t$ side, but is easier for
$f_i$, $i = 0, 1$, a simple application of monotone convergence. Next, as
in the proof of Theorem~\ref{discretePL1}, one can replace $g_t$ by 
$g_{t, N} = \min(g_t, N)$ and $f_i$ by $f_{i, N} = \max(f_i, -N)$: since
$g_{t, N}$ increases to $g_t$ as $N \to +\infty$, the argument for 
\[
  \int_{\Omega_2} 
   \log \Bigl( \int_{\R^n} e^{-g_{t, N}} \, d\gamma \Bigr) \, d\mu_2(t)
 \longto_N
  \int_{\Omega_2} 
   \log \Bigl( \int_{\R^n} e^{-g_t} \, d\gamma \Bigr) \, d\mu_2(t)
\] 
is the same as at the end of the proof of Theorem~\ref{genPL1}, and for 
$f_{i, N}$, decreasing to $f_i$, the reason is monotone convergence again
in $\int_{\R^n} e^{-f_{i, N}} \, d\gamma$, $i = 0, 1$. Then 
$0 \le g_{t, N} \le N$ and $f_{0, N}, f_{1, N} \ge - N$. As in the proof of
Theorem~\ref{discretePL1}, one can then assume that
$f_{0, N}, f_{1, N} \le 2 N / \beta$ with $\beta = \min(m, 1 - m)$.
Finally, approximation by convolution reduces to the case of bounded
continuous functions, and one can conclude by applying
Proposition~\ref{GaussianPropo}.
\end{proof}

\begin{proof}[Proof of Lemma~\ref{ExpoBoundLemma}]
We may assume that $f \ge 0$. For some integer $N \ge 0$, let 
$\varphi(x) = \min(f(x), N) \ge 0$. As in~\eqref{Fr}, define $\varphi_r$
by $e^{-\varphi_r} = P_{T - r} (e^{-\varphi})$, $0 \le r \le T$ and let
$\Phi_0 = \varphi_0(0)$, $F_0 = - \log P_{T} (e^{-f})(0) \ge \Phi_0 > 0$.
Note that $\varphi \ge 0$ implies $\varphi_r \ge 0$. Since $\varphi$ is
bounded and continuous, we may by Lemma~\ref{LBorell} consider the optimal
martingale corresponding to $\varphi$,
\[
   M_r 
 = \varphi_r \Bigl( B_r + \int_0^r u_\rho \, d\rho \Bigr)
    + \frac 1 2 \int_0^r |u_\rho|^2 \, d\rho
 = M_0 - \int_0^r u_\rho \cdot dB_\rho,
 \quad
 0 \le r \le T,
\]
with $M_0 = \Phi_0$, $u_r = - \nabla \varphi_r(X_r)$ and 
$X_r = B_r + \int_0^r u_\rho \, d\rho$. Define the square function
$(S_r)_{0 \le r \le T}$ of the martingale $(M_r - M_0)_{0 \le r \le T}$ by
\[
 S_r = \Bigl( \int_0^r |u_\rho|^2 \, d\rho \Bigr)^{1/2}.
\]
We know by~\eqref{converges} that $S_r$, $X_r$ and $M_r$ converge almost
surely and in $L^2$ to $S_T$, $X_T$ and $M_T$. Observe that $|u_r|$, $S_r$
and $M_r$ are bounded random variables for each fixed $r < T$. Consider the
exponential martingale $e^{\lambda M_r - \lambda^2 S_r^2 / 2}$ for 
$\lambda = 1/2$ and $0 \le r < T$, namely
\[
   \exp \bigl( M_r / 2 - S_r^2 / 8 \bigr)
 = \exp \bigl( \varphi_r(X_r) / 2 + S_r^2 / 4 - S_r^2 / 8 \bigr)
 = \exp \bigl( \varphi_r(X_r) / 2 + S_r^2 / 8 \bigr).
\]
By Fatou and the martingale property, we have that
\begin{align*}
     \E \exp \bigl( S_T^2 / 8 \bigr)
 &\le \E \exp \bigl( \varphi(X_T) / 2 + S_T^2 / 8 \bigr)
 \\
 &\le \lim_{r \to T}
       \E \exp \bigl( \varphi_r(X_r) / 2 + S_r^2 / 8 \bigr)
  =  e^{ \Phi_0 / 2 }.
\end{align*}
Let $Y_T = T^{-1/2} |X_T| \le T^{-1/2} |B_T| + S_T$ and observe that in
$\R^n$, one has $\E e^{|B_1|^2 / 4} = 2^{n/2}$. For every $\lambda > 0$,
using Cauchy--Schwarz and the inequality
$2 \sigma \tau \le c \sigma^2 + \tau^2 / c$ when $c, \sigma, \tau > 0$, we
can write
\begin{align}
     \Bigl( \E e^{ \lambda Y_T } \Bigr)^2
 &\le \Bigl( \E e^{ 2 \lambda T^{-1/2} |B_T| } \Bigr)
     \Bigl( \E e^{ 2 \lambda S_T } \Bigr)
  =  \Bigl( \E e^{ 2 \lambda |B_1| } \Bigr)
     \Bigl( \E e^{ 2 \lambda S_T } \Bigr)
 \label{toto}
 \\
 &\le e^{4 \lambda^2} \E e^{ |B_1|^2 / 4 } 
      e^{ 8 \lambda^2} \E e^{ S_T^2 / 8}
  =  2^{n / 2} e^{ 12 \lambda^2} \E e^{ S_T^2 / 8 }
 \notag
 \\
 &\le 2^{n / 2} e^{ 12 \lambda^2 + \Phi_0 / 2}
 \le e^{ 12 \lambda^2 + (F_0 + n) / 2 }.
 \notag
\end{align}
We have by the exponential bound~\eqref{ExponeBound} that
\[
     \delta
 :=  \E \bigl( f(X_T) - \varphi(X_T) \bigr)
 \le \E \bigl( \mathbf{1}_{ \{ f(X_T) > N \} } f(X_T) \bigr)
 \le a \E \bigl( \mathbf{1}_{ \{ M_T \ge N \} } e^{b |X_T|} \bigr).
\]
Using~\eqref{toto} with $\lambda = 2 b \sqrt T$, we get by
Cauchy--Schwarz and Markov 
\begin{equation} \label{delta}
     \delta
 \le a \sqrt {\frac {\E M_T} N} 
      \sqrt{ e^{24 b^2 T + (F_0 + n) / 4} }
 \le a \sqrt {\frac {F_0} N} e^{12 b^2 T + (F_0 + n) / 8}.
\end{equation}
This proves that for any given $\varepsilon > 0$, when $N$ is so large that
$\delta < \varepsilon$, the optimal drift $u$ for $\varphi$ gives
\begin{align*}
     & \E \Bigl[ f \Bigl( B_T + \int_0^T u_r \, dr \Bigr)
      + \frac 12 \int_0^T u^2_r \, dr \Bigr] - \varepsilon
 \\
 \le & \E \Bigl[ \varphi \Bigl( B_T + \int_0^T u_r \, dr \Bigr)
      + \frac 12 \int_0^T u^2_r \, dr \Bigr]
  =  \Phi_0
 \le F_0.
\end{align*}

 Conversely, since $1 \ge P_T e^{- \min(f, N)} \searrow P_T e^{-f}$ when 
$N \to +\infty$, we may start by choosing $N$ large enough, so that 
$F_0 < \Phi_0 + \varepsilon$. For any drift $u$ in $D_2$, the inequality
case for $\varphi$ gives
\begin{align*}
     F_0 - \varepsilon
  &<  \Phi_0 
 \le \E \Bigl[ \varphi \Bigl( B_T + \int_0^T u_\rho \, d\rho \Bigr)
          + \frac 1 2 \int_0^T |u_\rho|^2 \, d\rho \Bigr]
\\
 &\le \E \Bigl[ f \Bigl( B_T + \int_0^T u_\rho \, d\rho \Bigr)
              + \frac 1 2 \int_0^T |u_\rho|^2 \, d\rho \Bigr].
\end{align*}
This implies the inf formula~\eqref{borell} for the function $f$. Observe
that this final part does not use any upper bound on $f$, proving thus the
last claim of Lemma~\ref{ExpoBoundLemma}.
\end{proof}

\begin{lemma} \label{addingpsi}
Assume that $\mu$ is a finite measure on $(\Omega, \Sigma)$, $\nu$ a
probability measure on~$\R^n$, $(s, x) \mapsto f_s(x)$ a
$\Sigma \otimes \mathcal{B}_{\R^n}$-measurable function, and assume that 
$s \mapsto \log \Bigl( \int_{\R^n} e^{-f_s} \, d\nu \Bigr)$ is
$\mu$-integrable on $\Omega$. For every $\varepsilon > 0$, there exists a
continuous function $\psi \ge 0$ on $\R^n$, tending to $+\infty$ at
infinity, such that $0 \le \psi(x) \le |x|^2$ and
\[
      \int_\Omega \!
       {-} \log \Bigl(
               \int_{\R^n} e^{ {-} f_s(x) - \psi(x)} \, d\nu(x) 
              \Bigr) d\mu(s)
 {\le} \int_\Omega \!
      - \log \Bigl( \int_{\R^n} e^{ {-} f_s(x)} \, d\nu(x) \Bigr) d\mu(s)
       + \varepsilon.
\]
\end{lemma}

\begin{proof}
Consider $\chi(x) = |x|^2 / (1 + |x|^2)$, for $x \in \R^n$. Note that 
$0 \le \chi \le 1$, that $\chi$ tends to~$1$ at infinity, and that for
every $x$, $\chi(x / k)$ decreases to $0$ as $k \to +\infty$. Write
\[
    e^{- F(s)}
 := \int_{\R^n} e^{- f_s(x)} \, d\nu(x),
 \quad
 s \in \Omega,
 \quad \text{let} \quad
 I := \int_\Omega F(s) \, d\mu(s) < +\infty,
\]
and when $F(s) < +\infty$, let $u_k(s)$ be defined by
\[
    e^{- F(s) - u_k(s)}
 := \int_{\R^n} e^{- f_s(x) - \chi(x / k)} \, d\nu(x) 
 \longto_k e^{- F(s)}.
\]
Then $0 \le u_k(s) \le 1$ and $u_k$ converges $\mu$-almost everywhere 
to~$0$; since $\mu$ is finite, we can find $k_1 > 1$ such that
\[
   \int_\Omega
    - \log \Bigl( 
            \int_{\R^n} e^{- f_s(x) - \chi(x / {k_1}) } \, d\nu(x) 
           \Bigr) \, d\mu(s)
 = \int_\Omega \bigl( F(s) + u_{k_1}(s) \bigr) \, d\mu(s) 
 < I + \frac \varepsilon 2.
\]
In the same way, we can find by induction an increasing sequence $(k_j)$ of
integers such that for every integer $p \ge 1$,
\[
   \int_\Omega
    - \log \Bigl(
            \int_{\R^n} 
             \exp \bigl( - f_s(x) - \sum_{j=1}^p \chi(x / k_j) \bigr) 
              \, d\nu(x) 
           \Bigr) \, d\mu(s)
 < I + \sum_{j=1}^p 2^{-j} \varepsilon,
\]
and we can check that
\[ 
     \psi(x) 
  =  \sum_{j=1}^{\infty} \chi(x / k_j)
  =  \Bigl( \sum_{j=1}^{\infty} \frac 1 {k_j^2 + |x|^2} \Bigr) |x|^2 
 \le \Bigl( \sum_{j=1}^{\infty} k_j^{-2} \Bigr) |x|^2 
\]
does the job (note that $\sum_{j \ge 1} k_j^{-2} \le \pi^2 / 6 - 1 < 1$
because $k_1 > 1$).
\end{proof}

\begin{claim} \label{measurability}
The almost optimal drifts $\{ U(s) \}_{s \in \Omega_1}$
in~\eqref{DefiDrift} can be chosen to be $\Sigma_1$-measu\-rable with
respect to $s \in \Omega_1$.
\end{claim}

\begin{proof}
In the proof of Proposition~\ref{GaussianPropo} we may start, using
Lemma~\ref{addingpsi}, by replacing $f_s(x) \ge 0$ by $f_s(x) + \psi(x)$,
where $\psi$ tends to $+\infty$ at infinity, without changing much the
value of $\int_{\Omega_1}
 - \log \Bigl( \int_{\R^n} e^{- f_s(x)} \, d\gamma(x) \Bigr) \, d\mu_1(s)$ 
and without destroying the exponential bound~\eqref{exponebound} for
$f_s(x) + \psi(x)$. We shall thus assume that $\psi(x) \le f_s(x)$. Recall
that $T = \tau > 0$ is fixed. 

 Let $\delta > 0$ be given and let $e^{-f_{s, r}} = P_{T - r}(e^{-f_s})$ as
in~\eqref{Fr}. We may find a partition of~$\Omega_1$ in countably many
subsets $A_k \in \Sigma_1$, $k \in \N$, such that on the set $A_k$, the
exponential bound~\eqref{exponebound} for $f_s$ is uniform,
\[
 f_s(x) \le a_k e^{b_k |x|},
 \quad \text{and such that} \quad 
 f_{s, 0}(0) \le F_k,
 \quad s \in A_k.
\]
It is enough to prove the measurability on each set $A_k$ separately. We
can choose $N_k$ so large that 
\[
   a_k \sqrt {\frac {F_k} {N_k}} e^{12 b_k^2 T + (F_k + n) / 8}
 < \delta.
\]
By~\eqref{delta}, we may replace $f_s$, $s \in A_k$, by 
$\varphi_s = \min(f_s, N_k)$ with an error $< \delta$ in the estimate of 
$- \log (P_T e^{-f_s})(0)$. The ``almost optimal drift" $U(s) \in D_2$ for
$f_s$ is chosen equal to the optimal drift for $\varphi_s$.

 After this reduction, $\varphi_s$ is ``constant at infinity" since
$f_s \ge \psi$, hence $A_k \ni s \mapsto \varphi_s$ is a map to the
separable Banach space $Z$ of continuous functions on $\R^n$ tending to a
limit at infinity, equipped with the sup norm, and by the measurability
assumption~\MeasuCond, $s \mapsto \varphi_s(x)$ is $\Sigma_1$-measurable for
every $x \in \R^n$. By a classical theorem of Lusin ---that Hausdorff
topologies weaker than a Polish topology have the same Borel
$\sigma$-algebra--- this implies that $s \mapsto \varphi_s$ is
$\Sigma_1$-measurable from $\Omega_1$ to $Z$. Next, consider the four
mappings sending $f \in Z$ to $e^{-f}$, to $(f_r)_{0 \le r \le T}$ defined
by~\eqref{Fr}, and, for any given $\varepsilon \in (0, T)$, to 
$(\nabla f_r)_{0 \le r \le T - \varepsilon}$ and
$(\nabla^2 f_r)_{0 \le r \le T - \varepsilon}$. On every bounded subset $B$
of~$Z$, these mappings are Lipschitz from~$B$, equipped with the norm of
$Z$, respectively to~$Z$, to $C_b([0, T] \times \R^n)$, or to 
$C_b([0, T - \varepsilon] \times \R^n, \R^p)$, $p = n, n^2$. The Lipschitz
constants depend on $B$, and on $\varepsilon$ for the last two. The map
sending $f \in Z$ to the solution $X_{r, f}$ of the stochastic differential
equation~\eqref{sde} is continuous, in the precise sense that for every
bounded subset $B$ of $Z$, there is $\kappa = \kappa(B, \varepsilon)$ such
that for every $\omega \in E$,
\[
     \sup_{0 \le r \le T - \varepsilon}
      \bigl| X_{r, f}(\omega) - X_{r, g}(\omega) \bigr|
 \le e^{\kappa T} \|f - g\|_\infty,
 \quad
 f, g \in B.
\] 
Indeed, for every fixed $\omega$, we have the deterministic differential
equation in the $r$ variable
\[
   X'_{r, f}(\omega) - X'_{r, g}(\omega) 
 = - \nabla f_r (X_{r, f}(\omega)) + \nabla g_r (X_{r, g}(\omega)).
\]
Writing $X'_{r, f} - X'_{r, g}$ as
$ - (\nabla f_r (X_{r, f}) - \nabla f_r (X_{r, g}))
  - (\nabla f_r (X_{r, g}) - \nabla g_r (X_{r, g}))$, using the Lipschitz
properties mentioned above (implying that the second derivatives of $f_r$
are uniformly bounded for $f \in B$), we see that for every fixed 
$\theta > 0$, the function
$  D(r) 
 = \bigl( |X_{r, f}(\omega) - X_{r, g}(\omega)|^2 + \theta^2 \bigr)^{1/2}$ 
satisfies on $[0, T - \varepsilon]$ a differential inequality of the form
$D' \le \kappa (D + \|f - g\|_\infty)$, with $D(0) = \theta$. It follows
that the ``almost optimal" drift 
$U : (s, r) \mapsto - \nabla \varphi_{s, r}(X_{s, r})$ is
$\Sigma_1$-measurable.
\end{proof}

\begin{lemma} \label{L2selection}
Under \SettA, assume that $f_s \ge 0$ and that for some
$\varepsilon_0 > 0$, the function
$\Omega_1 \ni s \mapsto
 - \log \Bigl( \int_{\R^n} e^{-f_s(x) - \varepsilon_0 |x|^2} \, dx \Bigr)$ 
is $\mu_1$-integrable. There exists then $\alpha \in L^2(X_1, \R^n)$ such
that
\[
   \int_{\Omega_1} f_s(\alpha(s)) \, d\mu_1(s) 
 < +\infty.
\]
The inf-convolution $f_{s, k}$ in~\eqref{InfConvolu} is universally
measurable, and there exists a $\mu_1$-negligible set $N \in \Sigma_1$ such
that $(s, x) \mapsto f_{s, k}(x)$ is 
$\Sigma_1 \otimes \mathcal{B}_{\R^n}$-measurable on 
$(\Omega_1 \setminus N) \times \R^n$. It is possible to find for every
$\alpha \in L^2(X_1, \R^n)$ and every $\varepsilon > 0$ a measurable
selection $u(s)$ such that 
\[
   f_s \bigl( \alpha(s) + u(s) \bigr) + k |u(s)|^2
 < f_{s, k}(\alpha(s)) + \varepsilon,
 \quad
 s \in \Omega_1 \setminus N.
\]
\end{lemma}

\begin{proof}
Let
\[
   e^{-F(s)} 
 = \kappa \int_{\R^n} e^{-f_s(x) - \varepsilon_0 |x|^2} \, dx
 = \int_{\R^n} e^{-f_s(x) - \varepsilon_0 |x|^2 / 2} \, d\nu(x),
\]
where $\kappa = (2 \pi / \varepsilon_0)^{-n/2}$ is chosen so that
$\kappa \int_{\R^n} e^{- \varepsilon_0 |x|^2 / 2} \, dx = 1$, making $\nu$ a
probability measure. Consider the Borel set
\[
 C = \{ (s, x) : f_s(x) + \varepsilon_0 |x|^2 / 2 < F(s) + 1 \}
\]
in the Polish space $\Omega_1 \times \R^n$. The projection of this set on
$\Omega_1$ is an analytic set, and contains the Borel set 
$B := \{F < +\infty\}$. We have $\mu_1(\Omega_1 \setminus B) = 0$. By the
Jankoff--von Neumann selection theorem (see~\cite[Theorem~6.9.1]{Boga}), we
can find a universally measurable section 
$\sigma(s) = (s, \alpha(s)) \in C$ defined on~$B$. We get that
\[
 f_s(\alpha(s)) + \varepsilon_0 |\alpha(s)|^2 / 2 < F(s) + 1,
\]
from what the conclusions follow since $f_s \ge 0$. 

 Similarly, for every real $c$, the set $\{ (s, x) : f_{s, k}(x) < c \}$ is
the projection of the Borel set
\[
 \{ (s, x, u) : f_s(x + u) + k |u|^2 < c \}.
\]
It follows that for every fixed $x$, the function $s \mapsto f_{s, k}(x)$ is
universally measurable, hence $\mu_1$-equivalent to a Borel function on
$\Omega_1$. Let $D$ be a countable dense set in $\R^n$. For every 
$d \in D$, there is a negligible set $N_d \in \Sigma_1$ such that 
$s \mapsto f_{s, k}(d)$ is Borel outside $N_d$. Let $N = \bigcup_d N_d$.
Since $x \mapsto f_{s, k}(x)$ is continuous, we get that 
$(s, x) \mapsto f_{s, k}(x)$ is Borel on 
$\bigl( \Omega_1 \setminus N \bigr) \times \R^n$.

 Let $\alpha$ be Borel from $\Omega_1$ to $\R^n$. Then 
$\Omega_1 \setminus N$ is the projection of the Borel set
\[
 \{ (s, u) : s \notin N, \ 
             f_s(\alpha(s) + u) + k |u|^2 
              < f_{s, k}(\alpha(s)) + \varepsilon
 \}.
\]
Consider as before a universally measurable section $\sigma(s) = (s, u(s))$
defined on $\Omega_1 \setminus N$. Then $u(s)$ provides the promised
selection.
\end{proof}

\begin{rmk}
Let $(\Omega, \Sigma)$ be a measurable space and let $\nu$ be a probability
measure on $(\R^n, \mathcal{B}_{\R^n})$. Let $V$ be the set of bounded real 
$\Sigma \otimes \mathcal{B}_{\R^n}$-measurable functions $h(s, x) = h_s(x)$
satisfying that for every $\varepsilon > 0$, there exist two 
$\Sigma \otimes \mathcal{B}_{\R^n}$-measurable functions 
$\varphi(s, x) = \varphi_s(x)$, $\psi(s, x) = \psi_s(x)$, such that for
every $s \in \Omega$, $\varphi_s \le h_s \le \psi_s$, $\varphi_s$ is \usc\
on $\R^n$, $\psi_s$ is \lsc\ and 
$\int_{\R^n} ( \psi_s - \varphi_s ) \, d\nu < \varepsilon$. Then $V$ is a
vector space of functions on $\Omega \times \R^n$, containing constants,
stable by $\sup(h_1, h_2)$ and stable by pointwise convergence of uniformly
bounded sequences. Indeed, suppose that $h_k \in V$, $k \in \N$, and that
$h_k(s, x) \to h(s, x)$ pointwise with $|h_k(s, x)| \le 1$. For every $k$,
let $\psi_{s, k} \ge h_{s, k}$ be \lsc\ and 
$\int_{\R^n} ( \psi_{s, k} - h_{s, k} ) \, d\nu < 2^{-k}$. Then, for 
every~$s$, we have $\psi_{s, k} - h_{s, k} \to 0$ $\nu$-a.e. hence 
$\psi_{s, k} - h_s \to 0$ $\nu$-a.e. Next, for every $s$,
$\chi_{s, n} = \sup_{k \ge n} \psi_{s, k}$ is \lsc\ and decreases to $h_s$
$\nu$-a.e. Let
\[
 B_n = \{s \in \Omega : \int_{\R^n} \bigl( \chi_n(s, x)
                - h(s, x) \bigr) \, d\nu(x)
              < \varepsilon \}.
\]
Then $B_n \in \Sigma$ increases to $\Omega$. Define 
$\psi(s, x) = \chi_n(s, x)$ when $s \in B_n \setminus B_{n-1}$, and
similarly on the $\varphi$ side. We get that the pointwise limit $h$
belongs to $V$.

 It follows that $V$ is the space of all bounded $\mathcal{A}$-measurable
functions, for some sub-$\sigma$-algebra $\mathcal{A}$ of 
$\Sigma \otimes \mathcal{B}_{\R^n}$. Since $V$ contains all indicators of
products $B \times C$, $B \in \Sigma$, $C \in \mathcal{B}_{\R^n}$, it
follows that $\mathcal{A} = \Sigma \otimes \mathcal{B}_{\R^n}$. The result
applies then to all $\mathcal{A}$-measurable functions $h$ such that $h_s$
is bounded for every $s$, by cutting $\Omega$ into pieces $B_k \in \Sigma$
where the bound of $h_s$ is in $[k, k+1)$.

 Suppose that $h$ is bounded and $\ge 0$, and that 
$H(s) = \int_{\R^n} h(s, x) \, d\nu(x) > 0$ for every $s \in \Omega$.
Applying the result to the function $H(s)^{-1} h(s, x)$, we find a function
$\varphi(s, x)$ \usc\ in $x$ such that $0 \le \varphi \le h$ and
\[
   \log \Bigl( \int_{\R^n} \varphi(s, x) \, d\nu(x) \Bigr)
 > \log \Bigl( \int_{\R^n} h(s, x) \, d\nu(x) \Bigr) - \varepsilon
\]
for every $s$. Applying to $h(s, x) = e^{ - f_s(x) }$, we get an \lsc\ in
$x$ function $\psi_s(x)$ such that $f_s \le \psi_s$ and
$- \log \bigl( \int_{\R^n} e^{-\psi_s} \, d\nu \bigr)
 \le - \log \bigl( \int_{\R^n} e^{-f_s} \, d\nu \bigr) + \varepsilon$.
This shows that in Theorem~\ref{genPL1}, the function $f_s$ can be
assumed to be a \lsc\ function.
\end{rmk}

\end{document}